\documentclass[11pt]{amsart}
\usepackage{latexsym,enumerate}
\usepackage{amssymb}
\usepackage[colorlinks]{hyperref}

\newcommand{\QQ}{{\mathbb Q}}
\newcommand{\ZZ}{{\mathbb Z}}

\newcommand{\bC}{{\mathbf{C}}}
\newcommand{\bG} {\mathbf G}
\newcommand{\bN}{{\mathbf{N}}}
\newcommand{\bO}{{\mathbf{O}}}
\newcommand{\bZ}{{\mathbf{Z}}}

\newcommand{\cE} {\mathcal E}
\newcommand{\cS} {\mathcal S}

\newcommand{\Aut}{{{\operatorname{Aut}}}}
\newcommand{\Cl}{{\operatorname{Cl}}}
\newcommand{\Gal}{{{\operatorname{Gal}}}}
\newcommand{\Irr}{{{\operatorname{Irr}}}}
\newcommand{\Ker}{\operatorname{Ker}}
\newcommand{\Out}{{{\operatorname{Out}}}}
\newcommand{\St}{{\mathsf {St}}}
\newcommand{\Syl}{{\operatorname{Syl}}}
\newcommand{\Sym}{{\operatorname{Sym}}}
\newcommand{\Alt}{{\operatorname{Alt}}}

\newcommand{\GL}{\operatorname{GL}}
\newcommand{\PSL}{{\operatorname{PSL}}}
\newcommand{\SL}{{\operatorname{SL}}}
\newcommand{\Sp}{{\operatorname{Sp}}}

\newcommand{\parat}{{{p\operatorname{-ar}}}}
\newcommand{\qarat}{{{q\operatorname{-ar}}}}
\newcommand{\preg}{{{p\operatorname{-reg}}}}
\newcommand{\pareg}{{{p\operatorname{-areg}}}}
\newcommand{\trat}{{2\operatorname{-rat}}}
\newcommand{\prat}{{p\operatorname{-rat}}}
\newcommand{\tw}[1]{{}^#1}

\newtheorem{theorem}{Theorem}[section]
\newtheorem{lemma}[theorem]{Lemma}
\newtheorem{proposition}[theorem]{Proposition}
\newtheorem{corollary}[theorem]{Corollary}
\newtheorem{conjecture}[theorem]{Conjecture}
\newtheorem{question}[theorem]{Question}

\theoremstyle{definition}

\numberwithin{equation}{section}

\begin{document}

\title[Almost $p$-rational characters]
{On almost $p$-rational characters of $p'$-degree}

\author[N.\,N. Hung]{Nguyen Ngoc Hung}
\address{Department of Mathematics, The University of Akron, Akron,
  OH 44325, USA}
\email{hungnguyen@uakron.edu}

\author[G. Malle]{Gunter Malle}
\address{FB Mathematik, TU Kaiserslautern, Postfach 3049, 67653 Kaiserslautern,
  Germany}
\email{malle@mathematik.uni-kl.de}

\author[A. Mar\'{o}ti]{Attila Mar\'{o}ti}
\address {Alfr\'ed R\'enyi Institute of Mathematics, Re\'altanoda utca 13-15,
  H-1053, Budapest, Hungary}
\email{maroti.attila@renyi.hu}

\thanks{The second author gratefully acknowledges support by the Deutsche
 Forschungsgemeinschaft (DFG, German Research Foundation) -- Project-ID
 286237555-- TRR 195. The work of the third author on the project has
 received funding from the European Research Council (ERC) under the
 European Union's Horizon 2020 research and innovation programme
 (grant agreement No. 741420), and is supported by the Hungarian
 National Research, Development and Innovation Office (NKFIH) Grant
 No.~K132951, Grant No.~K115799 and Grant No.~K138828. We thank
 Gabriel Navarro for his insightful comments
 on an earlier version of the manuscript.}




\begin{abstract}
Let $p$ be a prime and $G$ a finite group. A complex character of $G$ is called
\emph{almost $p$-rational} if its values belong to a cyclotomic field
$\QQ(e^{2\pi i/n})$ for some $n\in\ZZ^+$ prime to $p$ or precisely divisible
by~$p$. We prove that, in contrast to usual $p$-rational characters, there are
always ``many'' almost $p$-rational irreducible characters in finite groups.
We obtain both explicit and asymptotic bounds for the number of almost
$p$-rational irreducible characters of $G$ in terms of~$p$. In fact, motivated
by the McKay--Navarro conjecture, we obtain the same bound for the number of
such characters of $p'$-degree and prove that, in the minimal situation, the
number of almost $p$-rational irreducible $p'$-characters of $G$ coincides with
that of $\bN_G(P)$ for $P\in\Syl_p(G)$. Lastly, we propose a new way to detect
the cyclicity of Sylow $p$-subgroups of a finite group $G$ from its character
table, using almost $p$-rational irreducible $p'$-characters and the blockwise
refinement of the McKay--Navarro conjecture.
\end{abstract}

\maketitle


\section{Introduction}

Let $G$ be a finite group and $p$ a prime. Recall that a character
$\chi$ of $G$ is called \emph{$p$-rational} if its values lie in a
cyclotomic field $\QQ_n:=\QQ(e^{2i\pi/n})$ for some positive integer
$n$ prime to $p$. Extending this notion, we say that $\chi$ is
\emph{almost $p$-rational} if the values of $\chi$ are in $\QQ_n$
for some $n$ divisible by $p$ at most once.

$p$-Rational characters appear in several contexts. One reason is
that they behave nicely with regard to character extensions, see
\cite[Theorems~6.30 and 11.32]{Isaacs1}. Another is that they are
pointwise fixed under the Galois group
$\Gal(\QQ_{|G|}/\QQ_{|G|_{p'}})$, where $|G|_{p'}$ is the $p'$-part
of $|G|$, and therefore they play an important role in problems
concerning Galois actions on irreducible characters and conjugacy
classes, see \cite[Chapters~3 and 9]{Navarro18}.

In this paper we show that almost $p$-rational characters also occur
naturally in group representation theory, but they are somewhat
richer and more interesting than $p$-rational characters in at least
one aspect: there are always ``many'' of them in finite groups. This
is in contrast to $p$-rationality: every odd-order $p$-group has a
unique $p$-rational irreducible character, for instance.

To conveniently state the results, let us use $\Irr_\parat(G)$ to
denote the set of almost $p$-rational irreducible characters of $G$.

\begin{theorem}   \label{theorem-general-bound}
 Let $G$ be a finite group and let $p$ be a prime dividing the order of $G$.
 Then $|\Irr_\parat(G)|\geq 2\sqrt{p-1}$. Moreover, equality occurs if and only
 if $p-1$ is a perfect square and $G$ is isomorphic to the Frobenius group
 $C_{p^n}\rtimes C_{\sqrt{p-1}}$ for some $n\in \ZZ^+$.
\end{theorem}

\begin{theorem}   \label{theorem-general-bound-p^2}
 There exists a universal constant $c>0$ such that for every prime~$p$ and
 every finite group $G$ having a non-cyclic Sylow $p$-subgroup,
 $|\Irr_\parat(G)|>c\cdot p$.
\end{theorem}

Besides the fact that the notion of almost $p$-rationality naturally
extends $p$-rationality and hence it is interesting in its own
right, there are two other motivations for our results. The first
comes from the results in \cite{HK1,Maroti16,HK,MS20} on bounding from below
the number of conjugacy classes, which is also the number of
irreducible characters, of a finite group. We show that similar
bounds hold for the number of irreducible characters with a
specific field of values, namely the field generated by roots of
unity of order with $p$-part at most $p$.

Another motivation comes from the celebrated McKay--Navarro
conjecture, which asserts that there exists a permutation
isomorphism between the actions of a certain subgroup of the Galois
group $\Gal(\QQ_{|G|}/\QQ)$ on the set of $p'$-degree
irreducible characters of $G$ and that of the normalizer $\bN_G(P)$
of some $P\in\Syl_p(G)$ (see \cite{Navarro04}), and therefore
produces a compatibility between the values of $p'$-characters of
$G$ and those of $\bN_G(P)$. The next result is based on
Theorem~\ref{theorem-general-bound} and the McKay--Navarro
conjecture. Here we use $\Irr_{p',\parat}(G)$ for the set of those
characters in $\Irr_\parat(G)$ with $p'$-degree, and prove that
$|\Irr_{p',\parat}(G)|$ is minimal (in terms of $p$) if and only if
$|\Irr_{p',\parat}(\bN_G(P))|$ is minimal, a result consistent with
the McKay--Navarro conjecture.

\begin{theorem}   \label{theorem-p'-degree}
 Let $G$ be a finite group, $p$ a prime dividing the order of $G$ and~$P$ a
 Sylow $p$-subgroup of $G$. Then $|\Irr_{p',\parat}(G)|\geq 2\sqrt{p-1}$.
 Moreover, the following are equivalent:
 \begin{itemize}
  \item[(i)] $|\Irr_{p',\parat}(G)|=2\sqrt{p-1}$;
  \item[(ii)] $|\Irr_{p',\parat}(\bN_G(P))|=2\sqrt{p-1}$;
  \item[(iii)] $P$ is cyclic and $\bN_G(P)$ is isomorphic to the Frobenius
   group $P\rtimes C_{\sqrt{p-1}}$.
 \end{itemize}
\end{theorem}

This improves the main result of \cite{Malle-Maroti} by bringing
character values into consideration. We remark also that recent work
\cite{Giannelli-Hung-Schaeffer-Rodriguez} on characters of
$\pi'$-degree with small cyclotomic field of values implies that,
for any pair $\{p,q\}$ of primes and any non-trivial group $G$, the
two sets $\Irr_{p',\parat}(G)$ and $\Irr_{q',\qarat}(G)$ always
contain a non-trivial common character. This somewhat indicates that
the sets $\Irr_{p',\parat}(G)$ are ``large'' (see
Theorem~\ref{thm:GHSV} for details), a phenomenon that is reinforced
in a clearer way by Theorem~\ref{theorem-p'-degree}.

Let $\Phi(P)$, as usual, denote the Frattini subgroup of a Sylow
$p$-subgroup $P$ of $G$. The conditions in
Theorems~\ref{theorem-general-bound},
\ref{theorem-general-bound-p^2}, and \ref{theorem-p'-degree} on $P$
being non-trivial and non-cyclic are equivalent to the conditions
$p\mid |P/\Phi(P)|$ and $p^2\mid |P/\Phi(P)|$, respectively. In view
of the McKay--Navarro conjecture and other local/global conjectures,
it is not surprising that the local group $P/\Phi(P)$ is a key
invariant that controls the global numbers $|\Irr_{p',\parat}(G)|$
and $|\Irr_\parat(G)|$. In fact, when $G$ is an abelian $p$-group we
have $|\Irr_\parat(G)|=|\Irr_{p',\parat}(G)|=|P/\Phi(P)|$.

In the next main result we make an attempt to push the bound in
Theorem~\ref{theorem-p'-degree} up to~$p$, with the help of the
(known) solvable case of the McKay--Navarro conjecture (due to
Dade).

\begin{theorem}   \label{theorem-general-bound-p^3}
 Let $G$ be a finite group, $p$ a prime and $P$ a Sylow $p$-subgroup of $G$.
 If $|P/\Phi(P)|\geq p^3$, then $|\Irr_{p',\parat}(G)|>p$ provided that any of
 the following two conditions holds.
 \begin{itemize}
  \item[(1)] $G$ is solvable and $p > 7200$; or
  \item[(2)] the McKay--Navarro conjecture is true and $p$ is sufficiently
   large.
 \end{itemize}
\end{theorem}

The conditions in Theorem~\ref{theorem-general-bound-p^3} on $G$ being solvable
and $p$ being large are perhaps superfluous but we are not able to remove
either of them at this time. On the other hand, the condition
$|P/\Phi(P)| \geq p^3$ is necessary. For every prime~$p$ there is a metacyclic
group $G$ such that $|P| =|P/\Phi(P)| = p$ and $|\Irr_\parat(G)|\leq p$.
Moreover, results in Section~\ref{section-affine-groups} show that there are
infinitely many primes $p$ for which there are examples of groups $G$ and $P$
with $|P/\Phi(P)| = p^{2}$ and $|\Irr_\parat(G)|\leq p$.


The proof of Theorem~\ref{theorem-general-bound-p^3} depends on a
result concerning the existence of a linear $p'$-group $H\leq
\GL(V)$, where $V$ is a finite vector space in characteristic $p$,
such that the class number $k(HV)$ of $HV$ is at most $p$, see
Theorems~\ref{2} and \ref{primes1}. This existence result may be of
independent interest and useful in other purposes, as discussed at
the beginning of Section~\ref{section-affine-groups}. In fact, our
results in Sections~\ref{section-affine-groups} and
\ref{section-affine-groups-2} point out a possible way to detect the
cyclicity of Sylow $p$-subgroups of a finite group from its
character table using almost $p$-rational $p'$-characters. This was
known only for $p=2,3$, in a recent work of Rizo, Schaeffer Fry, and
Vallejo \cite{RSV}.

The McKay--Navarro conjecture admits a blockwise refinement, which
is often referred to as the Alperin--McKay--Navarro conjecture, see
\cite[Conjecture~B]{Navarro04}. Let
\[\Irr_{p',\parat}(B_0(G)):=\Irr_{p',\parat}(G)\cap B_0(G),\] where
$B_0(G)$ is the principal $p$-block $G$. The refinement implies that
\[|\Irr_{p',\parat}(B_0(G))|=|\Irr_{p',\parat}(B_0(\bN_G(P)))|=|\Irr(\bN_G(P)/\Phi(P)\bO_{p'}(\bN_G(P)))|,\]
where $P\in\Syl_p(G)$, see Section~\ref{section-mckay-navarro} for
more details. As $\bN_G(P)/\Phi(P)\bO_{p'}(\bN_G(P))$ is a
semidirect product of the $p'$-group
$\bN_G(P)/P\bO_{p'}(\bN_G(P)))$ acting faithfully on the vector
space $P/\Phi(P)$, to characterize the cyclicity of $P$, one would
need to understand the values of the class numbers of these
semidirect direct products.

Note that if $\dim(V)=1$ then $k(HV)=e+\frac{p-1}{e}$, where $e=|H|
\mid (p-1)$. Our work on the values of class numbers of affine
groups seems to suggest that the class numbers $k(HV)$ with
$\dim(V)=1$ are distinguished from those with $\dim(V)>1$. We indeed
confirm this for $p$ sufficiently large. This observation and the
Alperin--McKay--Navarro conjecture lead us to the following deep
question on the connection between Galois automorphisms and cyclic
Sylow subgroups. Here $\cS_p:=\{e+\frac{p-1}{e}: e\in\ZZ^+, e\mid
p-1\}$.

\begin{question}\label{conj:cyclicity}
Let $G$ be a finite group and $p$ a prime dividing $|G|$. Is it true
that Sylow $p$-subgroups of $G$ are cyclic if and only if
$|\Irr_{p',\parat}(B_0(G))|\in \cS_p$?
\end{question}

An affirmative answer to Question~\ref{conj:cyclicity} would provide
an(other) answer to Brauer's Problem~12 \cite{Brauer63}, which asks
for information about the structure of Sylow $p$-subgroups of $G$
one can obtain from the character table of $G$. The problem has
inspired several interesting local/global results over the past two
decades, such as
\cite{Navarro-Tiep-Turull,Navarro-Tiep14,Navarro-Solomon-Tiep15,SchaefferFry,
Navarro-Tiep,Malle19}, to name a few. Note also that, when $p\leq
3$, $\cS_p=\{p\}$, and thus the statement is equivalent to: Sylow
$p$-subgroups of $G$ are cyclic if and only if
$|\Irr_{p',\parat}(B_0(G))|=p$, which is exactly what was shown in
\cite{RSV}.

Theorems~\ref{theorem-general-bound},
\ref{theorem-general-bound-p^2}, \ref{theorem-p'-degree},
and~\ref{theorem-general-bound-p^3} are proved in
Sections~\ref{section-explicit}, \ref{section-asymptotic},
\ref{section-p'-degree}, and \ref{section-proof-1.4}, respectively.
In the last Section~\ref{sect:last}, we make some remarks on
Question~\ref{conj:cyclicity} and answer it for $p$-solvable groups
with $p$ sufficiently large.

\section{The McKay--Navarro conjecture}   \label{section-mckay-navarro}

Let $\Irr(G)$ denote the set of all irreducible ordinary characters of a finite
group $G$, and let $\Irr_{p'}(G):=\{\chi\in\Irr(G):p\nmid \chi(1)\}$, where $p$
is a prime. The well-known McKay conjecture \cite{McKay} asserts that, for
every $G$ and every $p$, the number of $p'$-degree irreducible characters of
$G$ equals that of the normalizer $\bN_G(P)$ of a Sylow $p$-subgroup $P$ of
$G$. That is,
\[
|\Irr_{p'}(G)|=|\Irr_{p'}(\bN_G(P))|.
\]

Navarro proposed that there should be a bijection from
$\Irr_{p'}(G)$ to $\Irr_{p'}(\bN_G(P))$ that commutes with the
action of the subgroup $\mathcal{H}$ of the Galois group
$\Gal(\QQ_{|G|}/\QQ)$ consisting of those automorphisms that send
every root of unity $\xi\in\QQ_{|G|}$ of order not divisible by $p$
to $\xi^{q}$, where $q$ is a certain fixed power of $p$, see
\cite[Conjecture~9.8]{Navarro18} and also
\cite{Navarro04,Turull08,Navarro-Spath-Vallejo} for more updates.
This refinement of the McKay conjecture has now become the
McKay--Navarro conjecture, also known as the Galois-McKay conjecture.

We define the \emph{$p$-rationality level} of a character $\chi$ to
be the smallest nonnegative integer $\alpha:=\alpha_p(\chi)$ such
that the values of $\chi$ belong to the cyclotomic field
$\QQ_n:=\QQ(e^{2\pi i/n})$ for some $n$ divisible by $p^\alpha$.
Remark that $\chi$ is $p$-rational if and only if $\alpha_p(\chi)=0$
and $\chi$ is almost $p$-rational if and only if
$\alpha_p(\chi)\leq1$. (In a very recent paper \cite{Navarro-Tiep21},
Navarro and Tiep call the smallest positive integer $n$ such that
$\QQ(\chi)\subseteq \QQ_n$ the
\emph{conductor} of $\chi$, denoted by $c(\chi)$. The
$p$-rationality level of $\chi$ simply is the logarithm to the base
$p$ of the $p$-part of $c(\chi)$; that is
$\alpha_p(\chi)=\log_p(c(\chi)_p)$. We thank G. Navarro for pointing
out this connection to us.)

As $\Gal(\QQ_{|G|}/\QQ_{p^\alpha|G|_{p'}})$ is contained in
$\mathcal{H}$ for every $\log_p|G|_p\geq\alpha\in\ZZ^{\geq 0}$, the
McKay--Navarro conjecture implies that the number of $p'$-degree irreducible
characters at any level $\alpha$ in $G$ and $\bN_G(P)$ would be the
same:
\[
|\{\chi\in\Irr_{p'}(G): \alpha_p(\chi)=
\alpha\}|=|\{\theta\in\Irr_{p'}(\bN_G(P)): \alpha_p(\theta)=
\alpha\}|.
\]
Since every irreducible character of $\bN_G(P)$ of $p'$-degree has
kernel containing the commutator subgroup $P'$ of $P$ and every
irreducible character of $\bN_G(P)/P'$ is automatically of $p'$-degree,
we then have
\[
|\{\chi\in\Irr_{p'}(G): \alpha_p(\chi)\leq
1\}|=|\{\theta\in\Irr(\bN_G(P)/P'):\alpha_p(\theta)\leq 1\}|,
\]
which means that \[ |\Irr_{p',\parat}(G)|=|\Irr_\parat(\bN_G(P)/P')|.
\]

Now suppose for a moment that $|G|$ is divisible by $p$. Then $|\bN_G(P)/P'|$
is also divisible by $p$, and thus, by Theorem~\ref{theorem-general-bound} and
the conclusion of the previous paragraph, the McKay--Navarro conjecture implies
that the number of almost $p$-rational irreducible characters of $p'$-degree
of $G$ is at least $2\sqrt{p-1}$, as claimed in the Introduction.


The Alperin-McKay-Navarro conjecture refines further the
McKay-Navarro conjecture by considering blocks. In a similar way as
with the McKay-Navarro conjecture, it implies, for every block $B$
with a defect group $D$, that
\[|\Irr_0(B)\cap\Irr_{\parat}(B)|=|\Irr_0(b)\cap\Irr_{\parat}(b)|,\]
where $b$ is the Brauer correspondent of $B$ and $\Irr_0(B)$ and
$\Irr_{\parat}(B)$ respectively denote the sets of height zero
characters and almost $p$-rational characters in $B$. In particular,
for principal blocks, we would have
\[|\Irr_{p',\parat}(B_0(G))|=|\Irr_{p',\parat}(B_0(\bN_G(P)))|\]

Note that both the McKay--Navarro conjecture and its blockwise
refinement are known to be true for $p$-solvable groups, proved by
Turull \cite{Turull08,Turull13}, and for groups with a cyclic Sylow
$p$-subgroup, established by Navarro \cite{Navarro04}.

\section{An explicit bound for $|\Irr_\parat(G)|$}   \label{section-explicit}

In this section we prove Theorem~\ref{theorem-general-bound}. We do so by
relating almost $p$-rationality of characters and almost $p$-regularity of
conjugacy classes. This connection will be used in
Section~\ref{section-asymptotic} as well to achieve the asymptotic bound.

We start with the easier case $p=2$. Note that almost $2$-rational characters
are precisely $2$-rational characters.

\begin{lemma}
 Let $G$ be a finite group of even order. Then $|\Irr_\trat(G)|\geq2$ with
 equality if and only if $G$ is a non-trivial cyclic $2$-group.
\end{lemma}

\begin{proof}
The bound $|\Irr_\trat(G)|\geq 2$ is clear from Burnside's theorem
that groups of even order always possess a non-trivial rational
irreducible character. It is also clear that the number of
$2$-rational irreducible characters of a non-trivial cyclic $2$-group
is exactly $2$. It remains to show that if $|\Irr_\trat(G)|= 2$
then $G$ must be a non-trivial cyclic $2$-group. If $G$ is
non-solvable then it was shown in \cite[Lemma~9.2]{Hung-Maroti20}
that $|\Irr_\trat(G)|\geq 3$ and thus we are done.

So suppose that $G$ is solvable and $|\Irr_\trat(G)|= 2$. First we
have $\bO^{2'}(G)=G$ and moreover $G/\bO^2(G)$ is cyclic since
otherwise $|\Irr_\trat(G)|> 2$. We claim that $L:=\bO^2(G)$ is
trivial. Assume otherwise, then $G_1:=G/\bO^{2'}(L)$ is a semidirect
product of a $2$-group $A$ isomorphic to $G/L$ acting on a
non-trivial odd-order group $B$ isomorphic to $L/\bO^{2'}(L)$. Since
$A$ is cyclic, every $\theta\in\Irr(B)$ is extendible to the inertia
subgroup $I_{G_1}(\theta)$. In fact, by
\cite[Corollary~6.4]{Navarro18}, $\theta$ has a unique extension
$\chi\in\Irr(I_{G_1}(\theta))$ such that $\QQ(\chi)=\QQ(\theta)$.
Now by Clifford's correspondence we have that
$\chi^{G_1}\in\Irr(G_1)$ is $2$-rational, implying that
$|\Irr_\trat(G_1)|\geq 3$, a contradiction.

We have shown that $G$ is a $2$-group. If $G/\Phi(G)$ is elementary
abelian of 2-rank at least 2 then $G$ would have at least 4 rational
characters, a contradiction. We conclude that $G/\Phi(G)$ is cyclic,
which means that $G$ is cyclic as well.
\end{proof}

To bound the number of almost $p$-rational irreducible characters,
it is helpful to work with the dual notion for conjugacy classes,
namely {almost $p$-regular classes}. We therefore define the
\emph{$p$-regularity level} of a conjugacy class $g^G$ of $G$ to be
$\log_p(|g|_p)$, where $|g|_p$ is the $p$-part of the order of $g$.
Clearly a class $g^G$ is $p$-regular if its $p$-regularity level is
$0$. We say that $g^G$ is \emph{almost $p$-regular} if its level is
at most 1.

Let $\Cl_\preg(G)$ denote the set of $p$-regular classes and
$\Cl_\pareg(G)$ denote the set of almost $p$-regular classes of
$G$. We use $k(G)$ to denote the number of conjugacy classes of $G$.
Recall that $\Irr_\prat(G)$ and $\Irr_\parat(G)$ are the sets of
$p$-rational and almost $p$-rational, respectively, irreducible
characters of $G$. Finally, $n(G,X)$ denotes the number of orbits of
a group $G$ acting on a set $X$.

We observe that if the exponent of a finite group $G$ is not
divisible by $p^2$, then $k(G)=|\Cl_\pareg(G)|=|\Irr_\parat(G)|$.

The following fact will be used often in our proofs.

\begin{lemma}   \label{lemma-p-reg<p-rat}
 Let $G$ be a finite group and let $p$ be an odd prime. Then
 $|\Cl_\pareg(G)|\leq |\Irr_\parat(G)|$ and
 $|\Cl_\preg(G)|\leq |\Irr_\prat(G)|$.
\end{lemma}

\begin{proof}
If $|G|$ is not divisible by $p^2$ then all the classes of $G$ are
almost $p$-regular and all the characters of $G$ are almost
$p$-rational, and hence the lemma follows. Suppose $p^2\mid |G|$.
Consider the natural actions of the Galois group
$\mathrm{Gal}(\QQ_{|G|}/\QQ_{p|G|_{p'}})$ on classes and irreducible
characters of $G$. Note that this group is cyclic of order
$|G|_p/p$, and let $\sigma$ be a generator of the group. An
irreducible character of $G$ is almost $p$-rational if and only if
it is $\sigma$-fixed, while if a class of $G$ is almost $p$-regular
then the class is $\sigma$-fixed. The first inequality then follows by
Brauer's permutation lemma.

The second inequality is well-known and indeed can be proved similarly.
\end{proof}

\begin{lemma}   \label{lemma-p-reg-quotient}
 Let $N$ be a $p'$-group and $N\unlhd G$. Then
 \[
 |\Cl_\pareg(G)|\geq |\Cl_\pareg(G/N)|+n(G,\Cl_\pareg(N))-1,
 \]
 where $n(G,\Cl_\pareg(N))$ is the number of $G$-orbits on $\Cl_\pareg(N)$.
\end{lemma}

\begin{proof}
It is clear that the number of almost $p$-regular classes of $G$
inside $N$ is at least $n(G,\Cl_\pareg(N))$. Let $gN$ be an
element of $G/N$ of order not divisible by $p^2$. Let
$g=g_{p}g_{p'}=g_{p'}g_p$ where $g_p$ is a $p$-element and $g_{p'}$
is a $p'$-element. Then $gN=g_{p}N\cdot g_{p'}N=g_{p'}N\cdot g_pN$.
Now the order of $g_pN$ is not divisible by $p^2$, and thus
$g_p^pN=N$, which implies that $g_p^p=1$ by the assumption on $N$.
We have shown that if $(gN)^{G/N}$ is an almost $p$-regular class
then $g$ is an almost $p$-regular element of $G$. The lemma follows.
\end{proof}

Next we record a consequence of a recent result \cite{Hung-Maroti20}
on bounding the number of $\Aut(S)$-orbits on the set of $p$-regular
classes of a non-abelian finite simple group $S$.

\begin{lemma}   \label{lemma-HM20}
 Let $S$ be a non-abelian simple group of order divisible by a prime~$p$. The
 number of $\Aut(S)$-orbits on $p$-regular classes of $S$ is at least
 $2(p-1)^{1/4}$. Moreover, if this number is at most $2\sqrt{p-1}$ then
 $p\leq 257$ and $p^2\nmid |S|$.
\end{lemma}

\begin{proof}
This follows from \cite[Theorem 2.1]{Hung-Maroti20}.
\end{proof}

\begin{lemma}   \label{lemma-p-reg>2p-1}
 Let $G$ be a finite group having a non-abelian minimal normal subgroup $N$
 and $p$ an odd prime such that $p\mid |N|$ but $p\nmid |G/N|$. Then
 $|\Irr_\parat(G)|> 2\sqrt{p-1}$.
\end{lemma}

\begin{proof}
By hypothesis $N$ is isomorphic to a direct product of copies of a
non-abelian simple group, say $S$. Let $n$ be the number of
$\Aut(S)$-orbits on $p$-regular classes of $S$. First suppose that
there are $k>1$ simple factors in $N$. We then have
$|\Cl_\preg(G)|\geq {n+k-1 \choose k}\geq n(n+1)/2$. By
Lemma~\ref{lemma-HM20} we know that $n\geq 2 (p-1)^{1/4}$. Therefore
it follows that $|\Cl_\preg(G)|>2\sqrt{p-1}$ and we are done by
Lemma~\ref{lemma-p-reg<p-rat}. So we assume that $N$ is a
non-abelian simple group.

If $n>2\sqrt{p-1}$, then by the same arguments we are also done. So
we assume furthermore that $n\leq 2\sqrt{p-1}$. Using
Lemma~\ref{lemma-HM20} again, we know that $p$ is a prime divisor of
$|S|$ such that $p^2\nmid |S|$. Therefore, by the assumption, $p\mid
|G|$ but $p^2\nmid |G|$. It follows that
\[
|\Irr_\parat(G)|=k(G)\geq 2\sqrt{p-1},
\]
by \cite{Brauer42}. The equality occurs only when $G$ is the
Frobenius group $C_p\rtimes C_{\sqrt{p-1}}$ by
\cite[Theorem~1.1]{Maroti16}, which is not the case here. Thus we
have $|\Irr_\parat(G)|> 2\sqrt{p-1}$, and the proof is finished.
\end{proof}

We can now prove Theorem~\ref{theorem-general-bound} for odd $p$.

\begin{theorem}   \label{theorem-general-bound-class}
 Let $G$ be a finite group and $p\geq3$ a prime dividing the order of~$G$.
 Then $|\Irr_\parat(G)|\geq 2\sqrt{p-1}$. Equality occurs if and only if $p-1$
 is a perfect square and $G$ is isomorphic to the Frobenius group
 $C_{p^n}\rtimes C_{\sqrt{p-1}}$ for some $n\in\ZZ^+$.
\end{theorem}

\begin{proof}
Let $F_n:=C_{p^n}\rtimes C_{\sqrt{p-1}}$ (when, of course, $\sqrt{p-1}$ is an
integer).

First we prove that $|\Irr_\parat(F_n)|=2\sqrt{p-1}$. Let
$P:=P_n=C_{p^n}$. As every almost $p$-rational irreducible character
of $P$ has kernel containing $\Phi(P)$ and
$|\Irr(F_n/\Phi(P))|=|\Irr(F_1)|=2\sqrt{p-1}$, it is sufficient to
show that every $\chi\in\Irr_\parat(F_n)$ lies above an almost
$p$-rational irreducible character of $P$.

So assume that $\chi\in\Irr_\parat(F_n)$ lies above some non-trivial
$\theta\in \Irr(P)$. Let $a$ be a generator of $P$ and let
$\xi:=\theta(a)$. In particular, $\xi$ is a primitive $p^k$-root of
unity for some $k\in \ZZ^+$. Let
$\theta=\theta_1,\theta_2,\ldots,\theta_{\sqrt{p-1}}$ be distinct
$F_n$-conjugates of $\theta$. We have
\[
\chi(a)=\sum_{i=1}^{\sqrt{p-1}} \theta_i(a)=\sum_{i=1}^{\sqrt{p-1}}
\xi^{m^i},
\]
for some integer $m>1$ such that $m^{\sqrt{p-1}}\equiv 1\,(\bmod
|P|)$. We know that $\chi(a)\in \QQ_{p(p-1)}\cap \QQ_{|P|}=\QQ_p$,
and hence $\chi(a)$ is fixed under the cyclic group
$\Gal(\QQ_{p^k}/\QQ_p)$ (of order $p^{k-1}$). Also, the powers
$\xi^{m^i}$ ($1\leq i\leq \sqrt{p-1}$) are permuted by
$\Gal(\QQ_{p^k}/\QQ_p)$, and it follows that $\Gal(\QQ_{p^k}/\QQ_p)$
fixes at least one, and hence all, of $\xi^{m^i}$. We have shown
that $\xi\in\QQ_p$, which means that $\theta$ is almost
$p$-rational, as desired.

We now prove that if $\sqrt{p-1}\not\in \ZZ$ or $\sqrt{p-1}\in\ZZ$
but $G\not\cong F_n$ for all $n\in \ZZ^+$, then $|\Irr_\parat(G)|>
2\sqrt{p-1}$. Let $N$ be a minimal normal subgroup of $G$. By
induction we may assume that $p\mid |N|$ and $p\nmid |G/N|$, or
$\sqrt{p-1}\in\ZZ$ and $G/N\cong F_m$ for some $m\in\ZZ^+$.

Consider the case $p\mid |N|$ and $p\nmid |G/N|$. If $N$ is abelian
then the exponent of $G$ is not divisible by $p^2$ and so every
irreducible character of $G$ is almost $p$-rational. Therefore
$|\Irr_\parat(G)|=k(G)\geq 2\sqrt{p-1}$ by
\cite[Theorem~1.1]{Maroti16}, and moreover, the equality occurs if
and only if $\sqrt{p-1}\in\ZZ$ and $G\cong F_1$. The case $N$
non-abelian follows from Lemmas~\ref{lemma-p-reg>2p-1}.

Next we consider the case $G/N\cong F_m$ for some $m$. Then
$\Irr_\parat(G/N)=2\sqrt{p-1}$. If $N$ is non-abelian then $N$ is a
direct product of copies of a non-abelian simple group, say $S$. By
considering the restriction of the character labeled by $(n-1,1)$
from the symmetric group $\Sym(n)$ to the alternating group
$\Alt(n)$ ($n\neq 6$), the Steinberg character for simple groups of
Lie type (see \cite{Schmid92}), and checking \cite{Atl1} directly
for sporadic groups, we find that there exists a non-trivial
character $\theta\in\Irr(S)$ such that $\theta$ extends to a
rational-valued character of $\Aut(S)$. The tensor product of copies
of $\theta$ then extends to a rational character of $G$ by
\cite[Corollary~10.5]{Navarro18} and the tensor-induced formula
\cite[Definition~2.1]{Gluck-Isaacs}, which implies that $G$ has a
rational irreducible character whose kernel does not contain $N$. We
now have
\[|\Irr_\parat(G)|\geq |\Irr_\parat(G/N)|+|\Irr_\parat(G|N)|\geq
2\sqrt{p-1} +1,\] as desired.

So we may assume that $N$ is abelian and $G/N\cong F_m$. When $N$ is
a $p'$-group we have
\begin{align*}
|\Cl_\pareg(G)|&\geq |\Cl_\pareg(G/N)|+n(G,N)-1\\
&=2\sqrt{p-1}+n(G,N)-1
\end{align*}
by Lemma~\ref{lemma-p-reg-quotient}, and it follows immediately by
Lemma~\ref{lemma-p-reg<p-rat} that \[|\Irr_\parat(G)|\geq
|\Cl_\pareg(G)|>2\sqrt{p-1}\] since $n(G,N)\geq 2$.

We may now assume that $N$ is an elementary abelian $p$-group and
$G/N\cong F_m$. It follows that $G$ has a normal Sylow $p$-subgroup,
say $P$, and moreover, $G=PK$ is a semidirect product of a cyclic
group $K$ (of order $\sqrt{p-1}$) acting faithfully on $P$. We have
\[
|\Irr_\parat(G)|\geq |\Irr_\parat(G/\Phi(P))|=k(G/\Phi(P))
\]
since every irreducible character of $G/\Phi(P)$ is almost
$p$-rational. As above, since $G/\Phi(P)$ has order divisible by
$p$, we have $k(G/\Phi(P))\geq 2\sqrt{p-1}$ with equality if and
only if $G/\Phi(P)\cong F_1$. Thus we may assume that
$G/\Phi(P)\cong F_1$. In particular, $P/\Phi(P)$ is cyclic, and
therefore so is $P$. Recall that $K\cong C_{\sqrt{p-1}}$ acts
faithfully on $P$ and observe that the automorphism group of the
cyclic group $P$ of odd prime power order is cyclic. We conclude
that $G\cong F_n$ with $n=\log_p(|P|)$, and this finishes the proof.
\end{proof}

We have completed the proof of Theorem~\ref{theorem-general-bound}
for both $p=2$ and $p$ odd.

\section{Almost $p$-rational characters of $p'$-degree}   \label{section-p'-degree}

In this section we prove Theorem~\ref{theorem-p'-degree}, using
Theorem~\ref{theorem-general-bound}, the known cyclic Sylow case of
the McKay--Navarro conjecture, and some representation theory of
finite reductive groups.

\subsection{The case $G$ has cyclic Sylow}   \label{subsect:cyclic-Sylow}
We start with the case where Sylow $p$-subgroups of $G$ are cyclic.
As discussed in Section~\ref{section-mckay-navarro}, parts (i) and
(ii) of the second statement of Theorem~\ref{theorem-p'-degree} are
then equivalent and the first statement of
Theorem~\ref{theorem-p'-degree} is true for $G$ and $p$.

We now show that parts (ii) and (iii) in
Theorem~\ref{theorem-p'-degree} are equivalent. In fact, the
equality part of Theorem~\ref{theorem-general-bound} easily implies
that (iii) implies (ii).

Assume that $\sqrt{p-1}$ is an integer and
$|\Irr_{p',\parat}(\bN_G(P))|=2\sqrt{p-1}$. It follows that
\[
|\Irr_\parat(\bN_G(P)/P')|=|\Irr_{p',\parat}(\bN_G(P)/P')|\leq2\sqrt{p-1},
\]
where we recall that $P'$ is the commutator subgroup of $P$. Using
Theorem~\ref{theorem-general-bound}, we deduce that
$|\Irr_\parat(\bN_G(P)/P')|=2\sqrt{p-1}$, and moreover,
$\bN_G(P)/P'$ must be isomorphic to the Frobenius group
$F:=C_{p^n}\rtimes C_{\sqrt{p-1}}$ for some $n\in\ZZ^+$. It follows
that $P$ is cyclic and indeed $\bN_G(P)\cong P\rtimes
C_{\sqrt{p-1}}$, as stated.


\subsection{Reduction to a $p'$-order quotient}
Let $(G,p)$ be a counterexample to the theorem such that $|G|$ is as
small as possible. By the previous subsection, $G$ has minimal order
subject to the conditions $|\Irr_{p',\parat}(G)|\leq2\sqrt{p-1}$ and
the Sylow $p$-subgroups of $G$ are not cyclic. Let $N$ be a minimal
normal subgroup of $G$. Then $p\mid |N|$ by the minimality of $G$.

We claim that $p\nmid |G:N|$. Assume otherwise. Then we have
\[
|\Irr_{p',\parat}(G)|=|\Irr_{p',\parat}(G/N)|=2\sqrt{p-1}
\]
and the Sylow $p$-subgroups of $G/N$ are cyclic.

Suppose that $N$ is non-abelian, and let $S$ be a simple direct
factor of $N$. By \cite[Theorem~3.3]{Navarro-Tiep10}, there exists
an $\Aut(S)$-orbit $\mathcal{O}$ of non-trivial $p'$-degree
irreducible characters of $S$ such that $p\nmid |\mathcal{O}|$ and
every character in $\mathcal{O}$ extends to a $\QQ_p$-valued
character of its inertia subgroup in $\Aut(S)$. By
\cite[Proposition~3.1]{Navarro-Tiep10}, this orbit produces some
$\chi\in\Irr_{p'}(G)$ with $\QQ_p$-values and $N\nsubseteq
\Ker(\chi)$. This violates the above equality
$|\Irr_{p',\parat}(G)|=|\Irr_{p',\parat}(G/N)|$.

The following lemma finishes the proof of the claim.

\begin{lemma}\label{lemma-kernel}
 Let $N$ be a normal $p$-subgroup of $G$. Suppose that the Sylow $p$-subgroups
 of $G/N$ are cyclic but those of $G$ are not. Then there exists
 $\chi\in\Irr_{p',\parat}(G)$ whose kernel does not contain $N$.
\end{lemma}

\begin{proof}
Let $P$ be a Sylow $p$-subgroup of $G$. Since $P$ is
not cyclic, neither is $P/\Phi(P)$ and this implies that $P$ has at
least $p^2$ linear characters with values in $\QQ_p$. On the other
hand, as $P/N$ is cyclic, the principal character $\textbf{1}_N$ of
$N$ has at most $p$ extensions to $P$ with values in $\QQ_p$ by
Gallagher's theorem. We deduce that there exists
$\theta\in\Irr_\parat(P)$ such that $\theta(1)=1$ and $N\nsubseteq
\Ker(\theta)$.

Let $\sigma$ be the automorphism in $\Gal(\QQ_{|G|}/\QQ)$ that fixes
$p'$-roots of unity and sends every $p$-power root of unity to its
$(p+1)$th-power. Then $\sigma$ has $p$-power order and a character
of $G$ or $N$ is almost $p$-rational if and only if it is fixed by
$\sigma$. In particular, $\theta$ is $\sigma$-fixed.

Consider the induced character $\theta^G$ of degree $|G:P|$. Then
$\theta^G$ is also $\sigma$-fixed. If $\chi$ is an irreducible
constituent of $\theta^G$, we have
$[\chi^\sigma,\theta^G]=[\chi^{\sigma},(\theta^{G})^\sigma]=
[\chi,\theta^G]^{\sigma}=[\chi,\theta^G]$, and thus $\sigma$
permutes the irreducible constituents of $\theta^G$. Since $\sigma$
has $p$-power order and $\theta^G$ has $p'$-degree, we deduce that
$\sigma$ fixes at least one $p'$-degree irreducible constituent of
$\theta^G$. This constituent lies over $\theta$, and as $N\nsubseteq
\Ker(\theta)$, its kernel does not contain $N$, as desired.
\end{proof}

\subsection{Reduction to simple groups of Lie type in characteristic $\ell\neq p$}

We continue to work with a minimal counterexample $(G,p)$. By the
previous subsection, we know that, for every minimal normal subgroup
$M$ of $G$, we must have $p\nmid |G:M|$. We conclude that $G$ has a
unique minimal normal subgroup, say $N$, and furthermore, $p\nmid
|G/N|$. If $N$ is abelian then
$|\Irr_\parat(G)|=|\Irr_{p',\parat}(G)|\leq 2\sqrt{p-1}$ by the
It\^o--Michler theorem, violating Theorem~\ref{theorem-general-bound}
as Sylow $p$-subgroups of $G$ are not cyclic. Therefore $N$ is
isomorphic to a direct product of copies of a non-abelian simple
group, say $S$, of order divisible by $p$.

The following lemma, which is essentially due to Navarro and Tiep,
allows us to go back and forth between almost $p$-rational
characters of $N$ and those of $G$.

\begin{lemma}\label{lemma-NT}
Let $G$ be a finite group and $N\unlhd G$ such that $p\nmid |G:N|$.
Let $\theta\in\Irr(N)$ and let $\chi\in\Irr(G|\theta)$. Then
$\theta$ is almost $p$-rational if and only if $\chi$ is almost
$p$-rational.
\end{lemma}

\begin{proof}
The \emph{only if} implication is a consequence of
\cite[Lemma~5.1]{Navarro-Tiep}. We now prove the \emph{if}
implication. So assume that $\chi\in\Irr(G|\theta)$ is almost
$p$-rational.

Let $\theta=\theta_1,\theta_2,\ldots,\theta_t$ be all the
$G$-conjugates of $\theta$. In other words, the $\theta_i$ are all
of the irreducible constituents of $\chi_N$ by Clifford's theorem.
Let $\sigma$ be the same Galois automorphism as in the proof of
Lemma~\ref{lemma-kernel}. Then $\sigma$ permutes the $\theta_i$. But
since $t$ is prime to $p$ by hypothesis and $\sigma$ has $p$-power
order, there exists a $\theta_i$ that is $\sigma$-fixed, which
implies that all of the $\theta_i$ are $\sigma$-fixed.
\end{proof}

By Lemma~\ref{lemma-NT}, if all $p'$-degree irreducible characters
of $S$ are almost $p$-rational, then so are the $p'$-degree
irreducible characters of $G$, and thus
$|\Irr_{p',\parat}(G)|=|\Irr_{p'}(G)|$, which implies that $(G,p)$ is
not a counterexample by the main result of \cite{Malle-Maroti}. As
observed in the proof of \cite[Theorem~5.8]{Navarro-Tiep} (see also
\cite[Theorem~1.3]{Tiep-Zalesskii04}), when $p>2$, every irreducible
character of any alternating group, of any sporadic simple group
(including the Tits group $\tw2F_4(2)'$), or of any simple group of
Lie type in characteristic $p$, is almost $p$-rational. When $p=2$,
by \cite[Propositions~2.1--2.4]{Malle19}, irreducible characters of
these groups remain almost $p$-rational, except for 4 characters of
degrees $27$ and $351$ of $\tw2F_4(2)'$.

Thus, from now on, we may assume that $S\neq\tw2F_4(2)'$ is a simple group of
Lie type in characteristic $\ell$ different from~$p$, or $S=\tw2F_4(2)'$
with $p=2$.

\subsection{The case $p=2,3$}

Let us assume for a moment that $(S,p)\neq (\tw2F_4(2)',2)$, so that
the defining characteristic $\ell$ of $S$ is not $p$. First suppose
that $N<G$. Then $G$ has at least two $p$-rational irreducible characters
whose kernels contain $N$. On the other hand, the so-called Steinberg character
$\St_S$ of $S$ degree $\St_S(1)=|S|_{\ell}$ is extendible to a rational-valued
character of $\Aut(S)$ (see \cite{Schmid92}), and thus, as before,
$G$ has a rational-valued irreducible character $\chi$ that extends
$\St_S\times\cdots\times \St_S\in\Irr(N)$. We deduce that
$|\Irr_{p',\parat}(G)|\geq 3>2\sqrt{p-1}$, as desired.

We now suppose that $G=N$, and in fact, it suffices to suppose that
$G=S$. By \cite[Theorem~2.1]{Giannelli-Hung-Schaeffer-Rodriguez},
there exists $\textbf{1}_S\neq \chi\in\Irr(S)$ of
$\{\ell,p\}'$-degree such that $\QQ(\chi)\subseteq \QQ_\ell$ or
$\QQ(\chi)\subseteq \QQ_p$. In particular,
$\chi\in\Irr_{p'}(S)\setminus \St_S$ and $\chi$ is almost
$p$-rational and it follows again that $|\Irr_{p',\parat}(S)|\geq
|\{\textbf{1}_S,\chi,\St_S\}|= 3$.

We are left with the case $p=2$ and $S=\tw2F_4(2)'$. But a quick
inspection of the character table of $\tw2F_4(2)'$ reveals that it
has four rational-valued irreducible characters of degrees $1, 325,
351$, and $675$, all of which are extendible to
$\Aut(S)=\tw2F_4(2)$, and so the above arguments apply to this case
as well.

\subsection{Finishing the proof of Theorem~\ref{theorem-p'-degree}
(assuming Theorems~\ref{theorem-simple-groups} and
\ref{theorem-simple-groups2})}

We have shown that the counterexample $G$ has a unique minimal
normal subgroup $N$ with $p\nmid |G/N|$ and $N$ is isomorphic to a
direct product of $t$ copies of a simple group $S$ of Lie type in
characteristic $\ell$ with $\ell\neq p$ and $5\leq p\mid |S|$.

Assume first that Sylow $p$-subgroups of $S$ are non-cyclic. If
$N=S$ then it suffices to show that $|\Irr_{p',\parat}(G)|>
2\sqrt{p-1}$. On the other hand, if $t\geq 2$ then, by using
Lemma~\ref{lemma-NT} and the same arguments as in
\cite[\S3.2]{Malle-Maroti}, we deduce that $|\Irr_{p',\parat}(G)|\geq
k(k+1)/2$, where $k$ is the number of $\bN_G(S)$-orbits (here we
view $S$ as a simple factor of $N$) on $\Irr_{p',\parat}(S)$, and
therefore it suffices to show that there are at least $2(p-1)^{1/4}$
$\Out(S)$-orbits on $\Irr_{p',\parat}(S)$. In summary, in this case,
we wish to establish the following result on simple groups of Lie type.

\begin{theorem}\label{theorem-simple-groups}
Let $S\neq \tw2F_4(2)'$ be a simple group of Lie type and $p\geq5$ a
prime not equal to the defining characteristic of $S$ such that
Sylow $p$-subgroups of $S$ are non-cyclic. Let $S\leq H\leq \Aut(S)$
be an almost simple group such that $p\nmid |H/S|$. Then
\begin{itemize}
\item[(i)] $|\Irr_{p',\parat}(H)|> 2\sqrt{p-1}$.

\item[(ii)] There are at least $2(p-1)^{1/4}$
$H$-orbits on $\Irr_{p',\parat}(S)$.
\end{itemize}
\end{theorem}

For the case when Sylow $p$-subgroups of $S$ are cyclic, by
Subsection~\ref{subsect:cyclic-Sylow}, we may assume that the number
$t$ of copies of $S$ in $N$ is at least 2, and therefore we need to
establish the following.

\begin{theorem}\label{theorem-simple-groups2}
 Let $X$ be a finite group with a unique minimal subgroup $N=S^t$, where
where $S\neq \tw2F_4(2)'$ is a simple group of Lie type, and
$p\geq5$ a prime not equal to the defining characteristic of $S$
such that Sylow $p$-subgroups of $S$ are cyclic and non-trivial.
Suppose that $t\geq 2$ and $p\nmid |X:N|$. Then
$|\Irr_{p',\parat}(X)|> 2\sqrt{p-1}$.
\end{theorem}

Proofs of Theorems~\ref{theorem-simple-groups} and
\ref{theorem-simple-groups2}, as expected, rely on the
representation theory of finite groups of Lie type, and therefore
are deferred to the next section.

\section{Groups of Lie type}   \label{sect:groups-Lie-type}

In this section we prove Theorems~\ref{theorem-simple-groups} and
\ref{theorem-simple-groups2}, which were left off at the end of
Section~\ref{section-p'-degree}.

We consider the following setup. Let $\bG$ be a simple linear
algebraic group of adjoint type with a Steinberg endomorphism
$F:\bG\to\bG$. We consider the characters of the finite almost
simple group $G:=\bG^F$. For this, let $(\bG^*,F)$ be in duality
with $(\bG,F)$ (see \cite[Definition~1.5.17]{GM20}) and
$G^*:=\bG^{*F}$. According to Lusztig, there is a partition
$$\Irr(G)=\coprod_{s\in G^*/\sim}\cE(G,s)$$
into Lusztig series, where the union runs over a system of
representatives $s$ of semisimple conjugacy classes in $G^*$ (see
\cite[Theorem~2.6.2]{GM20}).

\begin{lemma}   \label{lem:lusz}
 In the above setting, let $\sigma\in\Gal(\QQ_{|G|}/\QQ_{p|G|_{p'}})$. Assume
 that $p\ge5$ and let $s\in G^*$ be an almost $p$-regular semisimple element.
 Then we have:
 \begin{enumerate}
  \item[\rm(a)] The Lusztig series $\cE(G,s)$ is $\sigma$-stable.
  \item[\rm(b)] The semisimple character in $\cE(G,s)$ is almost $p$-rational.
 \end{enumerate}
\end{lemma}

\begin{proof}
The first claim is well-known, see \cite[Proposition~3.3.15]{GM20}.
For the second, note that by the first $\cE(G,s)$ is
$\sigma$-stable. Since $|\bZ(\bG^F)|=1$, there is exactly one
semisimple character in $\cE(G,s)$ (see
\cite[Definition~2.6.9]{GM20}), and it is  uniquely distinguished
among all characters in $\cE(G,s)$ by having non-zero multiplicity
in the rational valued class function $\Delta_\bG$ from
\emph{loc.~cit.}. Thus it is $\sigma$-stable.
\end{proof}

Let $\sigma:\bG\to\bG$ be an isogeny commuting with $F$. Then there
exists a dual isogeny $\sigma^*:\bG^*\to\bG^*$ such that the
following holds (see \cite[Proposition~7.2]{Tay18}):

\begin{proposition}   \label{prop:fld}
 Let $s\in\bG^{*F}$ be semisimple. Then
 $$\tw\sigma\cE(G,s)=\cE(G,{\sigma^*}^{-1}(s)).$$
 In particular $\cE(G,s)$ is $\sigma$-stable if the class of $s$ is
 $\sigma^*$-stable.
\end{proposition}

We will employ this in case of Steinberg endomorphisms $\sigma$
commuting with $F$, which induce field automorphisms on $G$. In this
case, $\sigma^*$ induces also a field automorphism on~$G^*$.

The following strengthens \cite[Theorem~5.4]{Malle-Maroti} by
taking almost $p$-rationality into account:

\begin{theorem}   \label{thm:simple}
 Let $\bG$ be a simple exceptional group of adjoint type with a Steinberg
 endomorphism $F$ such that $S:=[\bG^F,\bG^F]$ is simple. Assume that Sylow
 $p$-subgroups of $S$ are non-cyclic and $p\geq 5$ is not the underlying
 characteristic of $\bG$. Then either
 $$|\cE(G,1)\cap\Irr_{p',\parat}(G)|\ge 2\sqrt{p-1}$$
 or
 $$|\Irr_{p',\parat}(G)|\ge 2g\sqrt{p-1}^3,$$
 where $g$ denotes the order of the group of graph automorphisms of
 $\bG$.
\end{theorem}

\begin{proof}
We follow the proof of \cite[Theorem~3.1]{Malle-Maroti}. There we
had shown that the analogous statement holds for $\Irr_{p'}$ in
place of $\Irr_{p',\parat}$. In particular, in those cases when the
characters constructed there happen to be almost $p$-rational, our
claim will follow automatically. Now note that all unipotent
characters of groups of Lie type $G:=\bG^F$ are almost $p$-rational
for all primes $p\ge3$. This follows, for example, from
\cite[Proposition~4.5.5]{GM20}. Thus, whenever only unipotent
characters are used in the proof of
\cite[Theorem~3.1]{Malle-Maroti}, we may conclude.

Since the Sylow $p$-subgroups of $\tw2B_2(q^2)$ and $^2G_2(q^2)$ for
all primes $p\ge5$ are cyclic, we need not consider these. If $p$ is
at most equal to the bound given in Table~2 of \cite{Malle-Maroti},
the characters used in the proof of
\cite[Proposition~5.4]{Malle-Maroti} are unipotent and so we are
done by our previous remark. If $p$ is larger than that bound, it
does not divide the order of the Weyl group of $G$ and so the Sylow
$p$-subgroups of $G$ are abelian \cite[Theorem~25.14]{MT}. Note that
we need not consider the primes~$p$ corresponding to the second set
of columns in \cite[Table~2]{Malle-Maroti} since for those $p$,
Sylow $p$-subgroups of $G$ are cyclic.

In the notation of \cite[Proposition~5.4]{Malle-Maroti}, a Sylow
$p$-subgroup of $G^*$ contains an elementary abelian subgroup $E$ of
order $p^{a_d}$ lying in a Sylow $d$-torus $S_d$ of $G^*$, where $d$
is the order of $q$ modulo $p$, and $G^*$-fusion of elements of $E$
is controlled by the relative Weyl group $W_d$ of $S_d$. (Here $q$
is the absolute value of all eigenvalues of $F$ on the character
group of an $F$-stable maximal torus of $\bG$.) Since Sylow
$p$-subgroups of $G^*$ are abelian, the centraliser of any $s\in E$
contains a Sylow $p$-subgroup of $G^*$. Thus the semisimple
character in the corresponding Lusztig series $\cE(G,s)$ has degree
prime to $p$ by the degree formula \cite[Corollary~2.6.6]{GM20}.

Now from the known values of $a_d$, the orders of $W_d$ and the
lower bounds on $p$ in \cite[Table~2]{Malle-Maroti} it is
straightforward to check that there are at least
$$2g(p-1)\sqrt{p-1}$$
conjugacy classes of such elements $s$ of order $p$ in $G^*$, where
$g\le2$ denotes the order of the group of graph automorphisms of
$\bG$. By Lemma~\ref{lem:lusz} for each such $s$ the semisimple
character in $\cE(G,s)$ is almost $p$-rational, of $p'$-degree by
what we said before, so we conclude.
\end{proof}

\begin{corollary}
 Theorem~{\rm\ref{theorem-simple-groups}} holds true for $S$ a simple
 exceptional group of Lie type.
\end{corollary}

\begin{proof} Note that we are done if $|\cE(G,1)\cap\Irr_{p',\parat}(G)|\geq 2\sqrt{p-1}$. So assume
otherwise. The proof of Theorem~\ref{thm:simple} indeed then shows
that $G$ then has at least $2g\sqrt{p-1}^3$ semisimple almost
$p$-rational $p'$-characters. Also, by Lemma~\ref{lemma-NT}, for
part~(i) of the theorem it suffices to find enough almost
$p$-rational $p'$-degree characters of $S$.

The diagonal automorphisms of $S$ permute only the semisimple
characters in a fixed Lusztig series and thus we obtain at least
$2g\sqrt{p-1}^3$ orbits of (semisimple) almost $p$-rational
$p'$-characters of $S$ under diagonal automorphisms. Now by
Proposition~\ref{prop:fld} any field automorphism of $G$ stabilises
$\cE(G,s)$ if the dual automorphism of $G^*$ fixes the class of~$s$.
Since $s$ considered in the proof of Theorem~\ref{thm:simple} has
order~$p$, it lies in an orbit of length at most $p-1$ under the
cyclic group of field automorphisms of $G^*$. Thus there are at
least $2\sqrt{p-1}$ orbits of $\Out(G)$ on $\Irr_{p',\parat}(S)$,
which implies the same bound for the number of $\Out(S)$-orbits on
$\Irr_{p',\parat}(S)$. The theorem follows by noting that we do
have strict inequality in part~(i) by taking unipotent characters
into account.
\end{proof}

\begin{proposition}\label{prop:classical-gps}
 Theorem~{\rm\ref{theorem-simple-groups}} holds true for $S$ a simple
 classical group.
\end{proposition}

\begin{proof}
By \cite[Proposition~5.5]{Malle-Maroti} and our introductory remarks
at the beginning of the proof of Theorem~\ref{thm:simple} we may
assume that $p$ does not divide the order of the Weyl group of $G$,
so $p$ is greater than the rank of $G$. Then a Sylow $p$-subgroup
$P$ of $G^*$ is homocyclic with $a\ge2$ factors, and the automiser
$W$ of $P$ acts as a wreath product $C_d\wr\Sym(a)$ for some
$d|(p-1)$, or a subgroup of index~2 thereof in groups of type $D_n$.
Note that $P$ contains an elementary abelian subgroup $E$ of order
$p^{a}$ and that the wreath product has $k(d,a)$ irreducible
characters, which is by definition the number of $d$-tuples of
partitions of $a$, see \cite[\S3]{Broue-Malle-Michel}.

First assume that $a=2$ and $S$ is not of type $D_n$. Then there are
at least $k(d,2)=(d^2+3d)/2$ unipotent characters of $S$, all of
which are $\Aut(S)$-invariant (see \cite[Theorem~2.5]{Malle08}), of
$p$-height zero corresponding to $\Irr(W)$. The number of
$G^*$-classes of non-trivial $p$-elements in $G^*$ is at least
$(p^2-1)/(2d^2)$ and the field automorphisms can fuse at most
$(p-1)/d$ of those, so we find at least $(p+1)/2dg \geq (p+1)/4d$
orbits of semisimple characters in $\Irr_{p',\parat}(S)$ under the
automorphism group. Together with the aforementioned $(d^2+3d)/2$
unipotent characters, we have proved part~(ii) of the theorem in
this case. For part~(i) we let $\overline{H}:=H/(H\cap G)$ (where
$H$ is the almost simple group given in
Theorem~\ref{theorem-simple-groups}), which can be viewed as a
subgroup of the abelian group of the field and graph automorphisms
of $S$. On one hand, the number of irreducible characters of $H$
lying over the $(d^2+3d)/2$ unipotent characters of $S$ exhibited
above is at least
\[
\frac{d^2+3d}{2}k(H/S)\geq
\frac{d^2+3d}{2}k(\overline{H})=\frac{d^2+3d}{2}|\overline{H}|,
\]
since every unipotent character of $S$ is fully extendible to $H$
(\cite[Theorem~2.4]{Malle08}). (Here we use $k(X)$ to denote the
conjugacy class number of $X$.) On the other hand, the number of
irreducible characters of $H$ lying over previously considered
semisimple characters of $S$ is at least
\[
\frac{p^2-1}{2d^2|\overline{H}|},
\]
as these semisimple characters are $G$-invariant and thus
$(G\cap H)$-invariant. Note that all these characters are almost
$p$-rational and of $p'$-degree by Lemma~\ref{lemma-NT}. We therefore
have
\[|\Irr_{p',\parat}(H)|\geq \frac{d^2+3d}{2}|\overline{H}|+\frac{p^2-1}{2d^2|\overline{H}|}
\geq 2\sqrt{\frac{p^2-1}{4}}>2\sqrt{p-1}\] since $p\geq 5$, as
required.

The case $a=2$ and $S$ is of type $D_n$ is argued similarly. Here
$S$ has at least $k(d,2)/2$ unipotent characters of $p'$-degree and
at least $(p^2-1)/d^2$ (non-trivial) semisimple characters coming
from semisimple elements in $E$.

So now let $a\ge3$ and first assume that $W=C_d\wr\Sym(a)$ and $S$
is not of the type $D_4$. There are $|\Irr(W)|=k(d,a)$ unipotent
characters of $S$ of $p'$-degree. Now $k(d,a)\ge 2da$ unless
$(d,a)=(2,3)$. Assume we are not in the latter case. Then we are
done whenever $p-1\le(da)^2$. Assume that $p-1>(da)^2$. Then the
elementary abelian $p$-subgroup $E\cong C_p^a$ of $S$ has at least
$(p^a-1)/(d^a\,a!)$ classes under the action of $W$, which is
\begin{align*}\frac{p^a-1}{d^a\,a!}&\ge(p-1)^{3/2}\frac{(da)^{2a-3}}{d^a\,a!}
  =(p-1)^{3/2}\frac{d^{a-3}a^{2a-3}}{a!}\\
  &\ge(p-1)^{3/2}\frac{a^{2a-3}}{a!}> 4(p-1)^{3/2}.\end{align*}
By Lemma~\ref{lem:lusz} this yields at least that many orbits of
almost $p$-rational $p'$-characters of $S$ under diagonal
automorphisms. Moreover, the group of field automorphisms has orbits
of length at most $(p-1)/d$ on this set of classes of elements of
order~$p$, the diagonal automorphisms do not decrease the number of
classes, and the group of graph automorphisms of $S$ has order at
most~2. Thus there are more than $2\sqrt{p-1}$ $\Aut(S)$-orbits on
$\Irr_{p',\parat}(S)$, as desired. When $(d,a)=(2,3)$ we have to
consider the case that $p=29,31$, but again the above inequality
suffices.

Assume that $W=C_d\wr\Sym(a)$ and $S$ is of the type $D_4$. Then one
must have $a=4$, and so $d=2$, since $a\geq 3$. Now $S$ has
$k(2,4)=20$ unipotent characters of $p'$-degree but $\Aut(S)$ has
two nontrivial orbits of length $3$ on unipotent characters
(\cite[Theorem~2.5]{Malle08}), and thus we are done if $2\sqrt{p-1}<
16$. Otherwise we just repeat the above arguments (with $g=6$) and check that
$(p^4-1)/(2^4\cdot 4!) > 12(p-1)^{3/2}$ to achieve the required bound.

Finally assume that $a\ge3$ and $W$ has index~2 in
$C_d\wr\Sym(a)$. Then necessarily $d$ is even. Here,
$|\Irr(W)|\ge 2da$ and we can argue as before, unless
$(d,a)=(2,3),(4,3),(2,4),(2,5)$. Note that in these cases, the
number of $W$-orbits on $E$ is at least $2\frac{p^a-1}{d^a\,a!}$,
and using the explicit value of $|\Irr(W)|$ we can again conclude.
\end{proof}

To prove Theorem~\ref{theorem-simple-groups2} we need the following
simple observation.

\begin{lemma}\label{lem:wreath}
Let $A$ be an abelian group and $H$ a subgroup of
$A\wr\Sym(t)$ for some $t\in\ZZ^+$. Then
\[\frac{|\Irr(H)|}{|H|}\ge \frac{1}{(t!)^2}.\]
\end{lemma}

\begin{proof}
The factor $|H|/|\Irr(H)|$ is the average of the squares of all
(irreducible) character degrees of $H$ and so is at most $b(H)^2$
where $b(H)$ is the largest character degree of $H$. But $b(H)$
divides $|H/(A^t \cap H)|$ by Ito's theorem (see
\cite[Theorem~6.15]{Isaacs1}) and $|H/(A^t \cap H)|\leq t!$ since
$H/(A^t \cap H)$ is isomorphic to the image of $H$ under the natural
homomorphism from $A\wr\Sym(t)$ to $\Sym(t)$, and
thus the lemma follows.
\end{proof}

We now prove Theorem~\ref{theorem-simple-groups2}, which is restated.

\begin{theorem}   \label{theorem-simple-groups2-repeated}
 Let $X$ be a finite group with a unique minimal subgroup $N=S^t$, where
 $t\in\ZZ^+$ and $S\neq \tw2F_4(2)'$ is a simple group of Lie type. Let
 $p\geq5$ be a prime not equal to the defining characteristic of $S$ such
 that Sylow $p$-subgroups of $S$ are cyclic and non-trivial. Suppose that
 $t\geq 2$ and $p\nmid |X:N|$. Then $|\Irr_{p',\parat}(X)|> 2\sqrt{p-1}$.
\end{theorem}

\begin{proof}
The assumptions on $N$ and $X$ imply that $X$ is a subgroup of
$\Aut(N)=\Aut(S)\wr \Sym(t)$.
\medskip

(1) First we assume that $S$ is of exceptional type. As before we
let $d$ be the order of $q$ modulo $p$, where $q$ is the size of the
underlying field of $S$, $S_d$ be a Sylow $d$-torus of $G^*$, and
$W_d$ the relative Weyl group of $S_d$. Since Sylow $p$-subgroups of
$S$ are cyclic, $S_d$ is cyclic of order $\Phi_d(q)$, and thus $W_d$
is cyclic, see \cite[pp.~260-261]{GM20}.

By $d$-Harish-Chandra theory (see \cite[\S3.5]{GM20}), there are at
least $|\Irr(W_d)|=|W_d|$ many unipotent characters of $G$ of
$p'$-degree, each of which restricts irreducibly to $S$. (Recall
that $G$ is the finite reductive group of adjoint type with
$S:=[G,G]$.) These unipotent characters of $S$ are all extendible to
$\Aut(S)$, by \cite[Theorems~2.4 and 2.5]{Malle08}. (When $S$ is
$G_2(3^f)$ or $F_4(2^f)$, the graph automorphism of order 2 does
fuse certain unipotent characters of $S$ but they are not
$p'$-degree, provided that Sylow $p$-subgroups of $S$ are cyclic,
see \cite[pp.~478--479]{Carter85}.) Recall that unipotent characters
are almost $p$-rational for odd $p$. It follows that there are at
least $|W_d|+1$ orbits of characters in $\Irr_{p',\parat}(S)$ under
$\Aut(S)$ (note that there is at least one orbit of semisimple
characters), and thus there are at least ${t+|W_d|\choose t}$
$X$-orbits on $\Irr_{p',\parat}(N)$, and therefore it follows from
Lemma~\ref{lemma-NT} that
\[|\Irr_{p',\parat}(X)|\geq {t+|W_d|\choose t}.\]
Note that $|W_d|\geq 4$ for all types and relevant values of $d$ and
$t\geq 2$ by the assumption, and therefore we are done if $t\geq
48^{1/4}(p-1)^{1/8}$ or if $|W_d|\geq 2(p-1)^{1/4}$. In fact, we are
also done if $p<{6\choose 2}^2+1=226$. So we assume that none of
these occur.

For each unipotent character $\theta$ of $p'$-degree of $S$, the
character
\[\psi_\theta:=\theta\times\theta\times\cdots\times\theta\in\Irr(N)\]
is fully extendible to $\Aut(N)$ (see
\cite[Corollary~10.5]{Navarro18}), and hence to $X$. By Gallagher's
lemma (see \cite[Corollary~6.17]{Isaacs1}), the number of
irreducible characters of $X$ lying over those $|W_d|$ characters
$\psi_\theta$ of $N$ is at least $|W_d|\cdot |\Irr(X/N)|$, which in
turns is at least
\[|W_d|\cdot|\Irr(X/(X\cap G^t))|.\]

Now a Sylow $p$-subgroup of $G^*$ contains a cyclic subgroup of
order $p$ and the number of $G^*$-conjugacy classes of non-trivial
$p$-elements in $G^*$ is at least $(p-1)/|W_d|$, and thus $G$ has at
least $(p-1)/|W_d|$ semisimple characters that are all almost
$p$-rational and of $p'$-degree, by Lemma~\ref{lem:lusz} and
Proposition~\ref{prop:fld}. These characters also restrict
irreducibly to (semisimple) characters of $S$ since $p\geq 5$ is
coprime to $|\bZ(G^*)|$. Therefore, $N$ has at least
$((p-1)/|W_d|)^t$ irreducible characters that are products of
non-trivial semisimple characters of copies of $S$. It follows that
the number of irreducible characters of $X$ lying over these
characters of $N$ is at least
\[
\frac{(p-1)^t}{|W_d|^t\cdot |X/(X\cap G^t)|}.
\]
This and the conclusion of the previous paragraph, together with
Lemmas~\ref{lemma-NT} and \ref{lem:wreath}, yield
\begin{equation}\label{eq:1}\begin{aligned}|\Irr_{p',\parat}(X)|&\geq |W_d|\cdot|\Irr(X/(X\cap
G^t))|+\frac{(p-1)^t}{|W_d|^t\cdot
|X/(X\cap G^t)|}\\
& \ge 2\sqrt{\frac{(p-1)^t}{|W_d|^{t-1}}\cdot \frac{|\Irr(X/(X\cap
G^t))|}{|X/(X\cap G^t)|}}\ge
2\sqrt{\frac{(p-1)^t}{(t!)^2|W_d|^{t-1}}}\\
& \ge
2\sqrt{p-1}\left(\frac{p-1}{t^2|W_d|}\right)^{(t-1)/2},\end{aligned}
 \end{equation}
which is larger than the required bound of $2\sqrt{p-1}$ by our
earlier assumptions that $t< 48^{1/4}(p-1)^{1/8}$, $|W_d|<
2(p-1)^{1/4}$, and $p\geq 226$.

\medskip

(2) We now consider the case when $S$ is of classical type. As seen
in the proof of Proposition~\ref{prop:classical-gps}, the number of
$p'$-degree unipotent characters of $S$ is at least $|\Irr(W)|$,
where $W$ is a certain cyclic group of order depending on $p,q$ and
the rank of $G$. Assume for now that $S$ is not of untwisted type $D_4$ so
that the group of the field and graph automorphisms of $S$ is
abelian and hence Lemma~\ref{lem:wreath} applies. The arguments for
exceptional types can also be used to achieve the desired bound.
Here the order of $W$ is always at least $2$ and it is
straightforward to show that
\[
\min\left\{{t+|W|\choose
t},2\sqrt{\frac{(p-1)^t}{(t!)^2|W|^{t-1}}}\right\}>2\sqrt{p-1}
\]
for all possibilities of $t, p$ and $W$.

Finally we assume that $S$ is of untwisted type $D_4$. The group of
graph automorphisms of $S$ is then $\Sym(3)$. Let $B$ be the
group of field and graph automorphisms of $S$ (which is the
direct product of the cyclic group of field automorphisms and
$\Sym(3)$), $H$ a subgroup of $B\wr\Sym(t)$, and set
$H_1:=H\cap B^t$. Using \cite[Lemma~2(i)]{Guralnick-Robinson} (in
the language of the so-called \emph{commuting probability} of finite
groups), we have
\[
\frac{|\Irr(H)|}{|H|}\ge \frac{|\Irr(H_1)|}{|H_1|}\cdot |H:H_1|^2\ge
\frac{|\Irr(B^t)|}{(t!)^2\,|B|^t}\geq \frac{1}{2^t\cdot (t!)^2},
\]
and therefore instead of the bound \eqref{eq:1} we now only have
\[|\Irr_{p',\parat}(X)|\geq 2\sqrt{\frac{(p-1)^t}{2^t\,|W|^{t-1}\,(t!)^2}}.\]
This and the bound $|\Irr_{p',\parat}(X)|\geq {t+|W|\choose t}$ are
again sufficient to reach the conclusion, with a notice that here
$|W|=3$ since Sylow $p$-subgroups of $S$ are cyclic and so $p$ must
divide $q^2\pm q+1$.
\end{proof}

We have completed the proofs of Theorems~\ref{theorem-simple-groups}
and \ref{theorem-simple-groups2}, and hence the proof of the main
Theorem~\ref{theorem-p'-degree}.

\section{An asymptotic bound for $|\Irr_\parat(G)|$}   \label{section-asymptotic}

In this section we prove Theorem~\ref{theorem-general-bound-p^2}.

We keep the notation introduced at the beginning of
Section~\ref{section-explicit}.

\begin{lemma}\label{lemma-c2}
There exists a constant $c_1 > 0$ such that if $G$ is any finite
group having an elementary abelian minimal normal subgroup $V$ of
$p$-rank at least $2$ and $p^2\nmid |G/V|$, then
$|\Irr_\parat(G)|>c_1\cdot p$.
\end{lemma}

\begin{proof}
By choosing $c_1<1/2$ if necessary, we assume that $p$ is odd. By
Lemma~\ref{lemma-p-reg<p-rat}, we then have $|\Irr_\parat(G)|\geq
|\Cl_\pareg(G)|$, and it follows that
\[
|\Irr_\parat(G)|\geq
\frac{1}{2}(|\Irr_\parat(G)|+|\Cl_\pareg(G)|).
\]
Recall that $p^2\nmid |G/V |$. Thus every irreducible character of
$G/V$ is almost $p$-rational, and therefore we have
\[|\Irr_\parat(G)|\geq |\Irr_\parat(G/V)|=k(G/V).\] On the other
hand, each $G$-orbit on $V$ produces at least one conjugacy class of
$G$ of elements in $V$, and thus
\[
|\Cl_\pareg(G)|\geq n(G,V),
\]
where $n(G,V)$ is the number of $G$-orbits on $V$. The three above
inequalities imply that
\[
|\Irr_\parat(G)|\geq \frac{1}{2}(k(G/V)+n(G,V)).
\]

It was shown in the proof of \cite[Proposition~2.2]{MS20} that,
under the same hypothesis, there exists a constant $c_1'>0$ such
that $k(G/V)+n(G,V)>c_1'\cdot p$. Now by choosing
$c_1:=\min\{c_1'/2,1/2\}$, we have the required bound
$|\Irr_\parat(G)|>c_1\cdot p$.
\end{proof}

For the next lemma, we denote by $M(S)$ the Schur multiplier of a
simple group $S$.

\begin{lemma}\label{lemma-c3}
There exists a constant $c_2>0$ such that for any non-abelian finite
simple group $S$ and any prime $p$ such that $p^2\mid |S|$ or $p\mid
|M(S)|$ or $p\mid |\Out(S)|$, we have
$n(\Aut(S),\Cl_{p'}(S))>c_2\cdot p$.
\end{lemma}

\begin{proof}
It is sufficient to assume that $p\geq 5$ and $S$ is an alternating
group or a finite group of Lie type.

Let $S=\Alt(n)$ for $n\geq 5$. The assumption on $p$ implies
that $p\leq n$. Observe that there are $(p-1)/2$ cycle types of odd
length up to $p-1$. Therefore $n(\Aut(S),\Cl_{p'}(S))\geq
(p-1)/2>p/3$.

Let $S$ be a simple group of Lie type defined over the field of
$q=\ell^f$ elements ($\ell$ is prime) with $r$ the rank of the
ambient algebraic group. By \cite[Theorem~1.4]{Hung-Maroti20}, we
have \[|\Cl_\preg(S)|>\frac{q^r}{17r^2}\] for every prime $p$.
Also, using the known information on $|\Out(S)|$ (see
\cite{Atl1} for instance), we find that there exists a constant
$c_{21}>0$ such that \[|\Out(S)|< c_{21}\cdot fr.\] It follows that
\begin{equation}\label{equation-3}
n(\Aut(S),\Cl_{p'}(S))\geq
\frac{|\Cl_\preg(S)|}{|\Out(S)|}>\frac{q^r}{17c_{21}fr^3}.
\end{equation}

First suppose that $p\mid |M(S)|$ or $p\mid |\Out(S)|$. Then $p\leq
\max\{r+1,f\}$. It follows from (\ref{equation-3}) that there exists
a constant $c_{22}>0$ such that $n(\Aut(S),\Cl_{p'}(S))>c_{22}\cdot
p$ for all possibilities of $S$ and $p$.

Next we suppose that $p^2\mid |S|$. It is an elementary result in
number theory that for $m,n\in\ZZ^+$ we have
$\gcd(q^m-1,q^n-1)=q^{\gcd(m,n)}-1$ and
$\gcd(q^m\pm1,q^n+1)=\gcd(2,q-1)$ or $q^{\gcd(m,n)}+1$. Assume that
$S$ is a classical group of rank at least 2. By inspecting the order
formulas, we observe that $|S|=\frac{1}{d(S)}q^{a(S)}P_S(q)$, where
$a(S)$ is the order of the group of diagonal automorphisms of $S$,
$a(S)$ is a suitable integer, and $P_S(q)$ is a polynomial in $q$
that can be written as a product of certain polynomials of the form
$q^i\pm 1$ with $i$ at most $r+1$. We deduce that $p\leq
q^{(r+1)/2}+1$ for all $S$ of classical type with rank $r\geq 2$. It
follows from the bound (\ref{equation-3}) that there exists a
constant $c_{23}>0$ such that $n(\Aut(S),\Cl_{p'}(S))>c_{23}\cdot p$
for all relevant $S$ and $p$. For $S=\PSL_2(q)$ we have $p\leq
(q+1)^{1/2}$ and the lemma also follows from the bound
(\ref{equation-3}). Similar arguments apply when $S$ is of
exceptional type different from $\tw2B_2(q)$ with $q=2^{f}$,
$^2G_2(q)$ with $q=3^{f}$, or $\tw2F_4(q)$ with $q=2^{f}$, where
$f\geq 3$ is an odd integer. For $S=\tw2B_2(q)$ we have $p=\ell$ or
$p\leq (q+\sqrt{2q}+1)^{1/2}$ and
we are also done by the bound (\ref{equation-3}). 
The cases $S=^2G_2(q)$
and $S=\tw2F_4(q)$ are similar and we skip the details.
\end{proof}

\begin{lemma}\label{lemma-c4}
There exists a constant $c_3>0$ such that for any non-abelian finite
simple group $S$ of order divisible by a prime $p$, we always have
$|\Cl_\preg(S)|>c_3\cdot \sqrt{p}|\Out(S)|$.
\end{lemma}

\begin{proof}
This essentially follows from the proof of Lemma~\ref{lemma-c3}. One
just uses the inequality (\ref{equation-3}) and obvious upper bounds
for a prime divisor of $|S|$.
\end{proof}

We are now in position to prove
Theorem~\ref{theorem-general-bound-p^2}.

\begin{theorem}\label{theorem-general-bound-p^2-repeated}
There exists a universal constant $c>0$ such that for every prime
$p$ and every finite group $G$ having a non-cyclic Sylow
$p$-subgroup, we have $|\Irr_\parat(G)|>c\cdot p$.
\end{theorem}

\begin{proof}
Let $G$ be a minimal counterexample with $c=\min\{c_1,c_2,c_3$,
$1/258\}$. First we observe that $p^2\nmid |G/N|$ for every
non-trivial $N\unlhd G$. Therefore, by Lemma~\ref{lemma-c2}, every
abelian minimal normal subgroup of $G$ must be isomorphic to a
cyclic group of order $p$. Suppose that there are more than one of
them, which implies that there are exactly two, say $A$ and $B$. Let
\[T:=A\times B,\, C:=\bC_G(T),\, H:=G/C,\, \text{and}\, K:=\bC_H(A).\] Note
that $H$ is an abelian group with exponent dividing $p-1$. Now
$k(G/T)\geq k(G/C)=|H|$. Moreover, the number of $K$-orbits on $A$
is $1+(p-1)/|K|$ and the number of $H/K$-orbits on $B$ is
$1+(p-1)/|H/K|$. We deduce that the number of $G$-orbits on $T$ is
at least
\[
\left(1+\frac{p-1}{|K|}\right)\left(1+\frac{p-1}{|H/K|}\right),
\]
which is greater than $1+(p-1)^2/|H|$. We therefore have
\[
k(G)\geq |H|+\frac{(p-1)^2}{|H|}\geq 2(p-1).
\]
Note that the exponent of $G$ is not divisible by $p^2$. So
$|\Irr_\parat(G)|=k(G)\geq 2(p-1)$, a contradiction.

We conclude that $G$ has at most one abelian minimal normal subgroup
and moreover, if there is one, it must be isomorphic to $C_p$.

First suppose that $G$ has no non-abelian minimal normal subgroup.
It then follows that $G$ has a unique abelian minimal normal
subgroup $N\cong C_p$, and furthermore, $p^2\nmid |G/N|$ but $p\mid
|G/N|$.

Since $G$ has non-cyclic Sylow $p$-subgroup, we have $N<G$. Let $M/N$ be
a minimal normal subgroup of $G/N$ such that $M\subseteq \bC_G(N)$.
Such an $M$ exists since $p\mid |G/N|$ and $|G/\bC_G(N)|\leq
|\Aut(C_{p})|=p-1$. We claim that $p^2\mid |M|$, and thus $p\nmid
|G/M|$. Assume otherwise, then $M= N\times K$ for some non-trivial
$p'$-subgroup $K$ of $M$ by the Schur--Zassenhaus theorem. This $K$
is characteristic in $M$, and thus normal in $G$, violating the fact
that $G$ has a unique abelian minimal normal subgroup $N$. The claim
follows.

Note that $M/N$ is a direct product of copies of a simple group, but
as $p^2\nmid |G/N|$ and $p\mid |G/N|$, there is only one such copy.
Suppose first that $M/N\cong C_p$. Then $M$ is a (normal) Sylow
$p$-subgroup of $G$. By the Schur--Zassenhaus theorem, we have $G=MK$
for some $p'$-subgroup $K$ of $G$. From the assumption that Sylow
$p$-subgroups of $G$ are non-cyclic, we must have $M\cong C_p\times
C_p$, and thus $M$ can be viewed as a (completely reducible)
$K$-module. In fact, since $K$ normalizes $N$, as a $K$-module $M$
is a direct sum of two $K$-modules of size $p$, and this contradicts
what we have shown above that $G$ has at most one abelian minimal
normal subgroup.

So $M/N$ must be isomorphic to a non-abelian simple group $S$, and
therefore $M$ is a quasisimple group with $\bZ(M)=N\cong C_p$. In
particular, $p$ is a divisor of the size of the Schur multiplier of
$S$. Therefore the number of $\Aut(S)$-orbits on $p$-regular classes
of $S$ is at least $c_2\cdot p$ by Lemma~\ref{lemma-c3}, implying
that $|\Cl_\pareg(G/N)|>|\Cl_\preg(G/N)|\geq c_2\cdot p$. Using
Lemma~\ref{lemma-p-reg<p-rat}, we then obtain
\[
|\Irr_\parat(G)|\geq |\Irr_\parat(G/N)|>c_2\cdot p\geq c\cdot p,
\]
a contradiction again.

Next we suppose that $G$ does have a non-abelian minimal normal
subgroup. We claim that every non-abelian minimal normal subgroup is
simple. Assume otherwise that $N=S^k$ is a minimal normal subgroup
of $G$ such that $S$ is non-abelian simple and $k\geq 2$. We then
have
\[
|\Cl_\pareg(G)|\geq|\Cl_\preg(G)|\geq n(n+1)/2,
\]
where $n$ is the number of $\Aut(S)$-orbits on $p$-regular classes
of $S$. By the definition of $c$, we know that $p>257$, and it
follows that $n>2\sqrt{p-1}$ by Lemma~\ref{lemma-HM20}. We now find
that
\[
|\Cl_\pareg(G)|\geq \sqrt{p-1}\left(2\sqrt{p-1}+1\right)>2p,
\]
and thus
\[
|\Irr_\parat(G)|>2p,
\]
by Lemma~\ref{lemma-p-reg<p-rat}. This is a contradiction.

The above arguments also apply when $G$ has more than one
non-abelian minimal normal subgroup. So we conclude that $G$ has
exactly one non-abelian minimal normal subgroup, and this is
isomorphic to a non-abelian simple group, say $S$. If $\bC_G(S)=1$
then $G$ is an almost simple group with socle $S$, and hence
$|\Cl_\preg(G)|>c_2\cdot p\geq c\cdot p$ by Lemma~\ref{lemma-c3},
which is again a contradiction.

Thus $\bC_G(S)$ has order divisible by $p$, but not $p^2$ by the
minimality of $G$. By a result of Brauer \cite{Brauer42}, we then
have $|\Cl_\pareg(\bC_G(S))|=k(\bC_G(S))\geq 2\sqrt{p-1}$. Also,
from Lemma~\ref{lemma-c4} we know  that $|\Cl_\pareg(S)|>c\cdot
\sqrt{p}|\Out(S)|$. Now
\begin{align*}
|\Cl_\pareg(\bC_G(S)\times S)|&=|\Cl_\pareg(\bC_G(S))|\times
|\Cl_\pareg(S)|\\
&>2c\sqrt{p(p-1)}|\Out(S)|.
\end{align*}
Note that $G/(\bC_G(S)\times S)$ can be viewed as a subgroup of
$\Out(S)$. It follows that
\begin{align*}
|\Cl_\pareg(G)|&\geq |\Cl_\pareg(\bC_G(S)\times S)|/|\Out(S)|\\
& >2c\sqrt{p(p-1)}\geq cp,
\end{align*}
which is again a contradiction, by Lemma~\ref{lemma-p-reg<p-rat}.
The proof is complete.
\end{proof}

\section{Affine groups with few conjugacy classes}   \label{section-affine-groups}

In this section let $H$ be a finite group acting faithfully on a
finite vector space $V$. Let $p$ be the prime divisor of the order
of $V$. Assume that the order of $H$ is not divisible by $p$. Let
$HV$ be the semidirect product of $H$ and $V$.

\subsection{Overview} \label{Sec:7.1}
In order to prove Theorem~\ref{theorem-general-bound-p^3}, we need
to classify all groups $HV$ with the property that the number
$k(HV)$ of conjugacy classes of $HV$ is at most $p$. For
Theorem~\ref{theorem-general-bound-p^3} we only need this
classification when $|V|\geq p^3$, but with
Conjecture~\ref{conj:cyclicity} and other purposes in mind, we
consider also the case $|V|\leq p^2$.

In fact, this problem was also mentioned to one of us by Gabriel
Navarro. The following example of Navarro is mentioned in the
paragraph after \cite[Lemma~1.3]{RSV}. If $p=11$, $H = \SL_2(5)$ and
$|V| = 11^2$, then $k(HV) = 10 < p$.

We begin with an asymptotic statement.

\begin{theorem}   \label{IulianAttila}
There exists a universal constant $c > 0$ such
that whenever $p \geq c$ is a prime, $V$ is a finite vector space of
order at least $p^3$ where $p$ is the characteristic of the
underlying field and $H$ is a finite group of order coprime to $p$
acting faithfully on $V$, then $k(HV) > p$.
\end{theorem}

\begin{proof}
This follows from an inspection of the proof of
\cite[Proposition~2.2]{MS20}.
\end{proof}

The purpose of this section is to work towards an explicit version
of Theorem~\ref{IulianAttila}. Whenever we can we will follow a
general approach to try and answer Navarro's question. At points we
need to impose explicit lower bounds on the prime $p$. Since even
the case of solvable groups $H$ seems difficult to handle, we will
not always treat the situation when $H$ has a quasisimple subnormal
subgroup.

In case $H$ is solvable (or more precisely when $H$ does not have a
quasisimple subnormal subgroup) an explicit value of the constant
$c$ in Theorem \ref{IulianAttila} is computed, see Theorem
\ref{primes1}. The list of solvable groups $HV$ with $k(HV) \leq p$
and $p$ sufficiently large is difficult to write down. Instead the
complete list of large primes $p$ is determined for which there are
groups $HV$ with $k(HV) \leq p$.

We accumulate most of the results in the following.

\begin{theorem}
\label{primes} There exists a universal constant $c > 0$ such that
the following is true. There is a vector space $V$ of order $p^{n}$
where $p \geq c$ is a prime and $n$ is a positive integer and there
is a subgroup $H$ of $\GL(V)$ of order coprime to $p$ such
that $k(HV) \leq p$ if and only if $n = 1$ or any of the following
holds.
\begin{itemize}
\item[(i)] $n = 2$ and $p \equiv 1 \pmod m$ for some even integer
$m$ with $12 \leq m \leq 36$.
\item[(ii)] $n = 2$ and $p \equiv 1 \pmod 5$ and there exists an integer $m$ dividing $p-1$ such that
$5 \leq m \leq 55$ and $(p-1)/m$ is even or $12 \leq m \leq 48$ and
$(p-1)/m$ is odd.
\end{itemize}
\end{theorem}

\begin{proof}
Set $c$ to be the maximum of $7 300 000$ and the $c$ in the
statement of Theorem~\ref{IulianAttila}. Let $p \geq c$ be a prime.
Let $V$ be a vector space of size $p^{n}$ for some positive integer
$n$. Let $H$ be a finite group of order coprime to $p$ acting
faithfully on $V$. If $n \geq 3$, then $k(HV) > p$ by
Theorem~\ref{IulianAttila}. We may thus assume that $n = 1$ or $n =
2$. For $n = 1$ the result follows from Lemma~\ref{metacyclic}. For
$n = 2$ the statement follows from Theorem~\ref{2}. 
\end{proof}

\subsection{Setup}
\label{Section 1}

Let $V$ be a vector space of size $p^n$ where $p$ is a prime and $n$
is a positive integer. Let $H$ be a subgroup of $\GL(V)$.
Assume that the size $|H|$ of $H$ is not divisible by $p$. Let $HV$
be the semidirect product of $V$ and $H$ with $V$ viewed as an
$H$-module.

Since $V$ is a completely reducible $H$-module by Maschke's theorem,
it may be written in the form $V = V_{1} \oplus \cdots \oplus V_s$
where $V_{1}, \ldots , V_{s}$ are (non-zero) irreducible
$H$-modules. For each $i$ with $1 \leq i \leq s$, write $V_{i}$ as a
sum $V_{i,1} + \cdots + V_{i,s_{i}}$ of subspaces $V_{i,j}$ of
$V_{i}$ in such a way that the set $\{ V_{i,1}, \ldots , V_{i,s_i}
\}$ is left invariant by $H$ and $s_i$ is the largest positive
integer with this property. Note that our condition ensures that for
every $i$ with $1 \leq i \leq s$ and every $j$ with $1 \leq j \leq
s_i$, the stabilizer $H_{i,j}$ in $H$ of the vector space $V_{i,j}$
acts primitively (but not necessarily faithfully) on $V_{i,j}$, that
is, $V_{i,j}$ is a primitive $H_{i,j}$-module. However, the
$V_{i,j}$ do not necessarily have the same size. Without loss of
generality, we assume that $p \leq |V_{1,1}| \leq \ldots \leq
|V_{s,1}|$.

The group $H$ acts on the set $\Omega$ of all subspaces $V_{i,j}$
with $1 \leq i \leq s$ and $1 \leq j \leq s_{i}$, that is, $H/B$ may
be viewed as a permutation group of degree $t = \sum_{i=1}^{s}
s_{i}$ where $B$ is the subgroup of all elements of $H$ which leave
every $V_{i,j}$ invariant.

The aim of this section is to describe pairs $(H,V)$ such that the
number $k(HV)$ of conjugacy classes of $HV$ is at most $p$.

\subsection{Basics}

For a finite group $X$ we denote the number of conjugacy classes of
$X$ by $k(X)$. We repeatedly use the bound $k(HV) \geq k(H) +
(|V|-1)/|H|$, which is a consequence of Clifford's theorem. This
estimate can be generalized \cite[Proposition~3.1b]{Schmid} as
follows.

\begin{proposition}
\label{prop} Let $\{ 1 = v_{1}, \ldots , v_{r} \}$ be a set of
representatives for the $H$-orbits on $V$. Then $k(HV) =
\sum_{i=1}^{r} k(C_{H}(v_{i}))$.
\end{proposition}

We remark that Proposition~\ref{prop} may be generalized to the
situation when the order of $H$ is divisible by $p$. See a result of
Guralnick and Tiep \cite[Corollary~2.5]{GurTie}.

\subsection{The metacyclic case}

\begin{lemma}   \label{metacyclic}
 Let $H$ be any subgroup of $\GL_1(p^{n}).n \leq \GL_n(p) = \GL(V)$.
 Then $k(HV) \leq p$ if and only if $n=1$.
\end{lemma}

\begin{proof}
If $n=1$, then $H = B$ is cyclic of order dividing $p-1$ and so
$$k(HV) = |H| + \frac{p-1}{|H|} \leq p.$$ Assume that $k(HV) \leq
p$. If $p=2$, then $n=1$ must hold. Assume also that $p \geq 3$.

Put $x = |H \cap \GL_1(p^{n})|$. We have
\begin{equation}
\label{previous} p \geq k(HV) \geq \frac{x}{n} + \frac{p^{n}-1}{x n}
= \frac{1}{n} \Big( x + \frac{p^{n}-1}{x} \Big) \geq \frac{2}{n}
\sqrt{p^{n}-1}
\end{equation}
by Clifford's theorem. This is a contradiction for $n \geq 3$ and $p\geq 3$.

Assume that $n = 2$. A more careful estimate as in (\ref{previous}) gives
$$p \geq k(HV) \geq 1 + \frac{x-1}{2} + \frac{p^{2}-1}{2x}.$$
Since the function $f(x) = \frac{x}{2} + \frac{p^{2}-1}{2x}$ defined
on the set $\{ 1, \ldots , p^{2}-1 \}$ takes its minimum at $x =
p-1$ and $x=p+1$ and this minimum is $p$, we arrive at a contradiction.
\end{proof}

\subsection{The case $p < 17$}

\begin{lemma}   \label{l1}
 Let $p < 17$. Then $k(HV) \leq p$ if and only if $n=1$ or $p=11$ and
 $HV = (C_{11} \times C_{11}):\SL_2(5)$.
\end{lemma}

\begin{proof}
If $n=1$ and $H$ is a subgroup of $\GL_1(p) = \GL(V)$, then $k(HV) \leq p$ by
Lemma~\ref{metacyclic}. If $p=11$ and $HV = (C_{11} \times C_{11}):\SL_2(5)$,
then $k(HV) = 10$ by \cite{VLVL1}.
If $n > 1$ and $HV \not= (C_{11}\times C_{11}):\SL_2(5)$, then $HV$ must have
more than $13$ conjugacy classes by the tables in \cite{VLVL1}, \cite{VLVL2}
and \cite{VLS}.
\end{proof}

\subsection{The case $n=2$}

\begin{lemma}   \label{primitiveandirreducible}
 Let $n=2$. If $k(HV) \leq p$, then $H$ is primitive and irreducible on $V$.
\end{lemma}

\begin{proof}
If $V$ is not an irreducible and primitive $H$-module, then $B$ is a
normal abelian subgroup (of exponent dividing $p-1$) inside $H$ of
index at most $2$, and so $$k(HV) \geq k(H) + \frac{p^{2}-1}{2|B|}
\geq 1 + \frac{|B|-1}{2} + \frac{p^{2}-1}{2|B|} \geq \frac{1}{2}
\Big( |B| + \frac{p^{2}}{|B|} \Big) > p,$$ since the function $f(x)
= x + p^{2}/x$ takes its minimum at $x = p$ in the interval
$[1,p^{2}]$ and $|B| \not= p$.
\end{proof}

\begin{lemma}   \label{k(H)}
 If $n=2$ and $k(HV) \leq p$, then
 $$k(H) +\frac{|V|-1}{|H|} \leq k(HV) \leq k(H) + \frac{|V|-1}{|H|} + 7200.$$
\end{lemma}

\begin{proof}
Let $n=2$ and $k(HV) \leq p$. The lower bound for $k(HV)$ follows
from Proposition~\ref{prop}. We proceed to establish the claimed
upper bound.

We know that $V$ is an irreducible and primitive $H$-module by
Lemma~\ref{primitiveandirreducible}. We also know that $H$ is not a
subgroup of $\GL_1(p^{2}).2$, for otherwise $k(HV) > p$ by
Lemma~\ref{metacyclic}. We then have $|H/Z(H)| \leq 60$ by \cite[pp.
213--214]{Huppert}.

We claim that there are at most $2 (|H/Z(H)|-1) \leq 118$ elements
$h \in H$ such that $|C_{V}(h)| = p$. For this it is sufficient to
see that every non-trivial coset of $Z(H)$ in $H$ contains at most
$2$ elements $h$ with $|C_{V}(h)| = p$. Let $C$ be an arbitrary
coset of $Z(H)$ in $H$ with an element $h$ fixing a non-zero vector
$v$ in $V$. Since the order of $h$ is not divisible by $p$, there is
a basis $\{ v, w \}$ in $V$ such that $w$ is an eigenvector of $h$
with eigenvalue $\lambda$ say. Observe that $Z(H)$ is a group of
scalars and that $C = h Z(H)$. If $\lambda = 1$, then $h=1$, $C =
Z(H)$ and there is no element $x$ in $C$ with $|C_{V}(x)| = p$.
Assume that $\lambda \not= 1$. Let $z$ be the scalar with
$\lambda^{-1}$ in the main diagonal. If $z \in Z(H)$, then $hz$ and
$h$ are the only elements in $C$ with a fixed point space of
dimension $1$. Finally, if $z \not\in Z(H)$, then $h$ is the only
element in $C$ with a fixed point space of dimension $1$. This
proves our claim.

Let $\{ 1 = v_{1}, \ldots , v_{r} \}$ be a set of representatives of
the $H$-orbits on $V$. We have $$k(HV) = \sum_{i=1}^{r}
k(C_{H}(v_{i})) = k(H) + \sum_{i=2}^{r} k(C_{H}(v_{i})) \leq k(H) +
\sum_{i=2}^{r} |C_{H}(v_{i})|$$ by Proposition \ref{prop}.

We claim that $\sum_{i=2}^{r} |C_{H}(v_{i})| \leq r + 7079$.

Each non-trivial element of $H$ that fixes a $v_i$ for some $i$ with
$2 \leq i \leq r$ can only fix at most $(p-1)/|Z(H)|$ elements of
the set $\{ v_{2}, \ldots, v_{r} \}$. It follows that
$\sum_{i=2}^{r} |C_{H}(v_{i})| \leq r-1 + 118 \cdot (p-1)/|Z(H)|$.

Observe that $|H| > p$ for otherwise $k(HV) > |V|/|H| \geq p$. Thus
$$(p-1)/|Z(H)| \leq |H|/|Z(H)| \leq 60$$ and so $\sum_{i=2}^{r}
|C_{H}(v_{i})| \leq r + 7079$, proving the claim.

We proceed to provide an upper bound for $r$. We proved that there
are at most $118$ elements $h \in H$ such that $|C_{V}(h)| = p$. It
thus follows from the orbit counting lemma and the inequality
$|H|>p$ that
$$r \leq \frac{1}{|H|} \sum_{h \in H} |C_{V}(h)| \leq \frac{1}{|H|}
\Big( |V| + 118 p + |H| \Big) < \frac{|V|}{|H|} + 119.$$

We conclude that $k(HV) \leq k(H) + r + 7079 < k(H) +
\frac{|V|}{|H|} + 119 + 7079$.
\end{proof}

\begin{lemma}
\label{l2} Let $n = 2$, $p > 270 000$ and $H$ solvable. Then $k(HV)
\leq p$ if and only if $H$ has a normal subgroup $Q$ which is a
quaternion group of order $8$, the group $X = Z(H)$ is cyclic and
$|X|$ divides $p-1$, the Fitting subgroup $F(H)$ is equal to $Q
\star X$ and $Q \cap X = Z(Q)$, the factor group $H/F(H)$ is
isomorphic to $\Sym(3)$, moreover if $x$ denotes $|X:Z(Q)|$ then $x$
divides $(p-1)/2$ and thus has the form $(p-1)/m$ for some even
integer $m$ with $12 \leq m \leq 36$.
\end{lemma}

\begin{proof}
We may assume that the group $H$ acts irreducibly and primitively on
$V$ by Lemma \ref{primitiveandirreducible}. We may also assume that
$H$ is not a subgroup of $\GL_1(p^{2}).2$ by
Lemma~\ref{metacyclic}.

We follow the argument of H\'ethelyi and K\"ulshammer \cite[p.~661--662]{HK}.
By \cite[Theorem 2.11]{ManzWolf}, we have $F(H) = Q
\star X$ and $Q \cap X = Z(Q)$, where $Q$ is a quaternion group of
order $8$, which is normal in $H$, and $X = Z(H)$ is cyclic and
$|X|$ divides $p-1$. Moreover, $H/F(H) \cong C_3$ or $H/F(H) \cong
\Sym(3)$. Let $x:= |X:Z(Q)|$, so that $x$ divides $(p-1)/2$ (and $p$
is odd). It is computed in \cite[p. 662]{HK} that $k(H) = 7x$ (in
which case $|H| = 24x$) if $H/F(H) \cong C_{3}$ and $k(H) = 8x$ (in
which case $|H| = 48x$) if $H/F(H) \cong \Sym(3)$. In the first case
the inequality $k(HV) \geq 7x + \frac{p^{2}-1}{24x} > p$ holds.
Assume that the second case holds. We thus have
\begin{equation}
\label{lowerupper} 8x + \frac{p^{2}-1}{48x} \leq k(HV) \leq 8x +
\frac{p^{2}-1}{48x} + 7200
\end{equation}
by Lemma \ref{k(H)}.

It is observed on \cite[p. 662]{HK} that the function $f: x \mapsto
8x + (p^{2}-1)/(48x)$ decreases for $x \leq (p-1)/20$ and increases
for $x \geq (p-1)/18$.

If $x \leq (p-1)/48$, then $k(HV) \geq 8 \cdot \frac{p-1}{48} +
\frac{p^{2}-1}{48 \cdot (p-1)/48} > p+1$. Assume that $x \geq
(p-1)/46$.

Let $x = (p-1)/m$ where $m$ is even and $2 \leq m \leq 46$. We may
rewrite (\ref{lowerupper}) as
$$\frac{8}{m} (p-1) + \frac{m}{48} (p+1) \leq k(HV) \leq
\frac{8}{m} (p-1) + \frac{m}{48} (p+1) + 7200$$ from which we get
\begin{equation}
\label{lowerupper2} \Big( \frac{8}{m} + \frac{m}{48} \Big) p - 4 <
k(HV) < \Big( \frac{8}{m} + \frac{m}{48} \Big) p + 7201.
\end{equation}

We have $\frac{8}{m} + \frac{m}{48} < 1$ if and only if $m$ is even
and $12 \leq m \leq 36$. Moreover, if $\frac{8}{m} + \frac{m}{48} <
1$, then $\frac{8}{m} + \frac{m}{48} < 0.973$, and so $k(HV) < 0.973
p + 7201$ by (\ref{lowerupper2}), which is at most $p$ provided that
$p > 270 000$.

It is easy to see that when $m$ is even (and in the range $2 \leq m
\leq 46$) and $\frac{8}{m} + \frac{m}{48} \geq 1$, then $\frac{8}{m}
+ \frac{m}{48} > 1.002$. In this case $p < 1.002 p - 4 < k(HV)$
since $p > 270 000$.
\end{proof}

\begin{lemma}
\label{5.4} Let $n=2$. The group $H$ is non-solvable if and only if
$H/Z(H)$ is isomorphic to $\Alt(5)$ and $p = 5$ or $p \equiv 1 \pmod
5$.
\end{lemma}

\begin{proof}
Let $H$ be a non-solvable subgroup of $\GL(V)$ having order
coprime to $p$. Then $H \leq Z(\GL(V)) \cdot \SL(V)$
by \cite[Theorem 3.5]{Giudici} and so $H/Z(H)$ is isomorphic to
$\Alt(5)$ and $p = 5$ or $p \equiv 1 \pmod 5$ by \cite[pp.
213--214]{Huppert}. The converse is clear.
\end{proof}

\begin{lemma}
\label{A_5} Let $n=2$ and $p > 7 300 000$. There is a non-solvable
subgroup $H$ of $\GL(V)$ (of coprime order) with $k(HV) \leq
p$ if and only if $p \equiv 1 \pmod 5$ and there exists an integer $m$ dividing $p-1$ such that
$5 \leq m \leq 55$ and $(p-1)/m$ is even or $12 \leq m \leq 48$ and $(p-1)/m$ is odd.
\end{lemma}

\begin{proof}
We may assume by Lemma \ref{5.4} that $p$ is a prime with $p \equiv
1 \pmod 5$, for if $p=5$ then $k(HV) > 5$ by Lemma \ref{l1}. Write
$H = Z(H) \star \SL_2(5)$ as the central product of the
cyclic group $Z(H)$ of order dividing $p-1$ and the perfect group
$\SL_2(5)$ of order $120$ with center of order $2$. There
are two possibilities for $H$; the intersection $Z(H) \cap
\SL_2(5)$ may have order $1$ or $2$. In the first case $k(H)
= k(\SL_2(5)) \cdot |Z(H)| = 9 \cdot |Z(H)|$ and in the
second case $k(H) = 9 \cdot (|Z(H)|/2)$ by \cite[Theorem
10.7]{Navarro18}. Define $c$ to be $9$ if $Z(H) \cap \SL_2(5)$
is trivial and $4.5$ if $|Z(H) \cap \SL_2(5)| = 2$. Observe
that $c=9$ if $|Z(H)|$ is odd and $c = 4.5$ if $|Z(H)|$ is even.

We have
\begin{equation}
\label{60} c \cdot |Z(H)| + \frac{p^{2}-1}{60 \cdot |Z(H)|} \leq
k(HV) \leq c \cdot |Z(H)| + \frac{p^{2}-1}{60 \cdot |Z(H)|} + 7200
\end{equation}
by Lemma \ref{k(H)}. Write $|Z(H)|$ in the form $(p-1)/m$ where $m$
is an integer. Inequality (\ref{60}) becomes
\begin{equation}
\label{60'} \frac{c}{m} (p-1) + \frac{m}{60} (p+1) \leq k(HV) \leq
\frac{c}{m} (p-1) + \frac{m}{60} (p+1) + 7200.
\end{equation}

If $m \geq 60$, then $k(HV) > p$ by Lemma \ref{l1}. Thus assume that
$m \leq 59$.

We distinguish two cases; when $(p-1)/m$ is odd and when $(p-1)/m$
is even.

Assume first that $(p-1)/m$ is odd. In this case $c=9$. If $m \leq
8$, then $k(HV) > p$ by (\ref{60'}). Thus assume that $m$ is an
integer with $9 \leq m \leq 59$. If $12 \leq m \leq 48$, then
$\frac{9}{m} + \frac{m}{60} \leq 0.9875$, while if $9 \leq m \leq
11$ or $49 \leq m \leq 59$ then $\frac{9}{m} + \frac{m}{60} >
1.0003$. Inequalities (\ref{60'}) imply $k(HV) < 0.9875 \cdot p +
7201$ in the first case and $1.0003 \cdot p - 1 < k(HV)$ in the
second case. If $p > 580 000$, then $k(HV) < p$ in the first case
and $p < k(HV)$ in the second.

Assume that $(p-1)/m$ is even. In this case $c = 4.5$. If $m \leq
4$, then $k(HV) > p$ by (\ref{60'}). Thus assume that $m$ is an
integer with $5 \leq m \leq 59$. If $5 \leq m \leq 55$, then
$\frac{4.5}{m} + \frac{m}{60} < 0.999$, while if $56 \leq m \leq 59$
then $\frac{4.5}{m} + \frac{m}{60} > 1.01$. Inequalities (\ref{60'})
imply $k(HV) < 0.999 \cdot p + 7201$ in the first case and $1.01
\cdot p - 1 < k(HV)$ in the second case. If $p > 7 300 000$, then
$k(HV) < p$ in the first case and $p < k(HV)$ in the second.
\end{proof}

\begin{theorem}
\label{2} Let $n=2$ and $p > 7 300 000$. There exists a subgroup $H$
of $\GL(V)$ (of coprime order) such that $k(HV) \leq p$ if
and only if any of the following holds for the prime $p$.
\begin{itemize}
\item[(i)] $p \equiv 1 \pmod m$ for some even integer $m$ with $12 \leq m \leq 36$.
\item[(ii)] $p \equiv 1 \pmod 5$ and there exists an integer $m$ dividing $p-1$ such that
$5 \leq m \leq 55$ and $(p-1)/m$ is even or $12 \leq m \leq 48$ and $(p-1)/m$ is odd.
\end{itemize}
\end{theorem}

\begin{proof}
There exists a solvable subgroup $H$ of $\GL(V)$ such that
$k(HV) \leq p$ if and only if (i) holds. This follows from Lemma
\ref{l2}. There exists a non-solvable subgroup $H$ of
$\GL(V)$ such that $k(HV) \leq p$ if and only if (ii) holds.
This follows from Lemma~\ref{A_5}.
\end{proof}

We remark that there are infinitely many primes $p$ congruent to $1$
modulo any given integer not divisible by $p$ by Dirichlet's theorem
on arithmetic progressions.

\section{Affine groups with few conjugacy classes: the case $n\geq 3$}   \label{section-affine-groups-2}

We continue to classify groups $H$ acting coprimely and faithfully
on a finite vector space $V$ of order $p^n$ ($n\geq 3$) with the
property that $k(HV)\leq p$.

We use the setup at the beginning of
Section~\ref{section-affine-groups}.

\subsection{Some general restrictions}
For an index $i$ with $1 \leq i \leq s$, let $K_{i}$ denote the
(faithful) action of $H$ on $V_{i}$ and $B_{i}$ denote the kernel of
the action of $K_{i}$ on the set $\{ V_{i,1}, \ldots , V_{i,s_{i}}
\}$.

\begin{lemma}
\label{l4} Let $i$ be an index with $1 \leq i \leq s$. We have
$k(K_{i}V_{i}) \leq p$ and $k(K_{i}/B_{i}) < p$.
\end{lemma}

\begin{proof}
Let $V_{i'}$ denote the sum of all $V_{j}$ with $j \not= i$. We have
$$p \geq k(HV) \geq k(HV/V_{i'}) \geq k(HV_{i}) \geq k(K_{i}V_{i})$$
since $HV_{i}$ modulo the kernel of the action of $H$ on $V_{i}$ is
$K_{i}V_{i}$. We also have $p \geq k(K_{i}V_{i}) > k(K_{i}/B_{i})$
since $B_{i}V_{i} \triangleleft K_{i}V_{i}$.
\end{proof}

For an index $i$ with $1 \leq i \leq s$ and an index $j$ with $1
\leq j \leq s_{i}$, the group $H_{i,j}$ was defined to be the
stabilizer in $H$ of the vector space $V_{i,j}$.

\begin{lemma}
\label{l45} For any $i$ and $j$ with $1 \leq i \leq s$ and $1 \leq j
\leq s_{i}$, the number of orbits of $H_{i,j}$ on $V_{i,j}$ is at
most $p$.
\end{lemma}

\begin{proof}
Fix $i$ and $j$. The number $r_{i,j}$ of orbits of $H_{i,j}$ on
$V_{i,j}$ is independent from $j$. Clearly, $H$ and so $K_{i}$ has
at least $r_{i,j}$ orbits on $V_{i}$. Thus $k(K_{i}V_{i}) \geq
r_{i,j}$ by Proposition~\ref{prop}. We conclude that $r_{i,j} \leq
p$ by Lemma~\ref{l4}.
\end{proof}

\subsection{Maximal primitive linear groups}

Let $i$ and $j$ be arbitrary indices with $1 \leq i \leq s$ and $1
\leq j \leq s_{i}$. Let $K_{i,j}$ be the factor group of $H_{i,j}$
modulo the kernel of the action of $H_{i,j}$ on $V_{i,j}$. The
vector space $V_{i,j}$ is a faithful, primitive and coprime
$K_{i,j}$-module. The goal of this subsection is to show, in as many
cases as possible, that the group $K_{i,j}$ has more than $p$ orbits
on $V_{i,j}$, under the assumption $p \geq 17$.

Forgetting the indices $i$ and $j$, let $W$ be a faithful, primitive
and coprime $P$-module over the prime field of size $p$ for some
finite group $P$. Moreover, assume that $P$ is a maximal subgroup of
$\GL(W)$ subject to these conditions. There is a (unique and
maximal) $\mathbb{F}_{p^{k}}$-vector space structure $W =
W_{d}(p^{k})$ on $W$ for some integers $d$ and $k$ with $n = dk$,
where $\mathbb{F}_{p^{k}}$ is the field of order $p^{k}$, such that
$P \leq \mathrm{\Gamma L}(d,p^{k})$, by \cite[Theorem~5.1(3)]{DHP}.

If $d=1$, then $k=n$ and $P = \mathrm{\Gamma L}(1,p^{n})$. In this
case $P$ has only $2$ orbits on $W$.

Assume that $d \geq 2$. We aim to show, in as many cases as
possible, that $P$ has more than $p$ orbits on $W$, provided that $p
\geq 17$. For this we will use the structure theorem given by Duyan,
Halasi, Podoski in \cite[Theorem~5.1]{DHP} for the $P$-module $W$.

\begin{lemma}
If $n \geq 3$, $d \geq 2$, $p > 3^{6}$ and $P$ has no component,
then $P$ has more than $p$ orbits on $W$.
\end{lemma}

\begin{proof}
Let $Z = Z(\GL(W))$. We may assume by \cite[Theorem~5.1(6)]{DHP}
that $(P \cap \GL(W))/Z$ has a unique minimal normal subgroup
$RZ/Z$, where $R$ is an $r$-group of symplectic type for some prime
$r$. (The group $R$ has the property that all of its characteristic
abelian subgroups are cyclic.) The factor group $R/Z(R)$ has size
$r^{2a}$ for some positive integer $a$. The vector space $W$ is an
absolutely irreducible $\mathbb{F}_{p^{k}}R$-module of dimension $d
= r^{a}$ where $k$ is as before the statement of the lemma. Let the
full normalizer of $R$ in $\mathrm{\Gamma L}(W)$ be $N$. Observe
that $P \leq N$. Put $M = N \cap \GL(W)$. The group $M/R Z$ can be
considered as a subgroup of the symplectic group $\Sp_{2a}(r)$.
Moreover, $|N/M| \leq d$.

We have $|P| \leq p^{k} \cdot r^{2a^{2} + 4a}$. If $p > 3^{6}$, $n
\geq 3$ and $d \geq 2$, then it is easy to see that $p^{k} \cdot
r^{2a^{2} + 4a} < p^{k \cdot r^{a} - 1} = |W|/p$, so $P$ has more
than $p$ orbits on $W$.
\end{proof}

\subsection{Permutation groups}

\begin{lemma}
\label{l25} Let $R$ be a permutation group of degree $r$. If $R$ has
no alternating composition factor of degree larger than $d \geq 4$,
then $|R| \leq d!^{(r-1)/(d-1)}$. Moreover, if $R$ is a primitive
permutation group not containing the alternating group $\Alt(r)$, then
$|R| < 24^{(r-1)/3}$.
\end{lemma}

\begin{proof}
The first statement is \cite[Corollary 1.5]{attila2002}. The second
statement is the fifth sentence in \cite[Section 4]{attila2002}.
\end{proof}

For a non-negative integer $d$ we denote the number of partitions of
$d$ by $\pi(d)$. For a finite group $S$ let $k^{*}(S)$ be the number
of orbits of $\mathrm{Aut}(S)$ on $S$.

\begin{lemma}
\label{l3} Let $R$ be a permutation group of degree $r$. Let $p$ be
a prime at least $17$ such that $k(R) \leq p$.
\begin{itemize}
\item[(i)] If $R$ has an alternating composition factor $S$ of
degree $d$ at least $5$, then $p > \pi(d)/(2 \cdot |\mathrm{Out(S)}|)$.

\item[(ii)] $|R| \leq {(\log p)}^{2(r-1)}$.
\end{itemize}
\end{lemma}

\begin{proof}
Let $S \cong \Alt(d)$ be an alternating composition factor of $R$
with $d \geq 5$. It follows from \cite[Lemma~2.5]{Pyber} that $p
\geq k(R) \geq k^{*}(S)$. Thus $p \geq k^{*}(S) >
k(S)/|\mathrm{Out}(S)|$. Since $S$ is a normal subgroup of index $2$
in the symmetric group of degree $d$, it follows that $k(S) \geq
\pi(d)/2$ where $\pi(d)$ denotes the number of partitions of $d$.
Statement (1) follows.

Assume again that $R$ has an alternating composition factor of
degree $d \geq 5$. From statement (1) and \cite[Corollary~3.1]{M} it
follows that $p > \pi(d)/8 \geq e^{2 \sqrt{d}}/112$. This gives
$$d \leq \frac{{(\log p)}^{2}}{8} + 1.7 \log p + 6 < {(\log p)}^{2}$$
for $p \geq 17$. Applying Lemma~\ref{l25}, we get
$$|R| \leq d!^{(r-1)/(d-1)} \leq d^{r-1} < {(\log p)}^{2(r-1)}.$$
Finally, if $R$ has no non-abelian alternating composition factor,
then $$|R| \leq 24^{(r-1)/3} \leq {(\log p)}^{2(r-1)}$$ by
Lemma~\ref{l25}.
\end{proof}

\subsection{The case $|V_{1,1}| = p \geq 17$ and $s_{1} > 1$}
\label{Section 9}

Note that $B_{1}$ is an abelian normal subgroup (of exponent
dividing $p-1$) of $K_{1}$.

\begin{lemma}
\label{l5} If $|V_{1,1}| = p \geq 17$, then $$|K_{1}/B_{1}| \cdot (p
- k(K_{1}/B_{1})) \geq 2 \cdot {(|V_{1}|-1)}^{1/2} - 1 >
{|V_{1}|}^{1/2}.$$
\end{lemma}

\begin{proof}
Note that $k(K_{1}V_{1})$ is the number of complex irreducible
characters of the group $K_{1}V_{1}$. This is at least $k(K_{1})$
plus the number of orbits of $K_{1}$ on $\mathrm{Irr}(V_{1})
\setminus \{ 1_{V_1} \}$ by Clifford's theorem where $1_{V_{1}}$
denotes the trivial character of $V_{1}$. Thus
\begin{equation}
\label{ee1} p \geq k(K_{1}V_{1}) \geq k(K_{1}) +
\frac{|V_{1}|-1}{|K_{1}|}
\end{equation}
by Lemma \ref{l4}. Since $B_{1}$ is an abelian normal subgroup of
$K_{1}$, we also have
\begin{equation}
\label{ee2} k(K_{1}) \geq k(K_{1}/B_{1}) +
\frac{|B_{1}|-1}{|K_{1}/B_{1}|},
\end{equation}
again by Clifford's theorem. Inequalities (\ref{ee1}) and
(\ref{ee2}) give
$$p \geq k(K_{1}/B_{1}) + \frac{|B_{1}|-1}{|K_{1}/B_{1}|} + \frac{|V_{1}|-1}{|K_{1}|},$$ that is,
\begin{equation}
\label{ee3} |K_{1}/B_{1}| \cdot (p - k(K_{1}/B_{1})) \geq |B_{1}|-1
+ \frac{|V_{1}|-1}{|B_{1}|}.
\end{equation}
The real valued function $f(x) = x + \frac{|V_{1}|-1}{x}$ defined on
the interval $[1,|V_{1}|-1]$ takes its minimum at $x =
\sqrt{|V_{1}|-1}$. We get
$$|K_{1}/B_{1}| \cdot (p - k(K_{1}/B_{1})) \geq 2 \cdot {(|V_{1}|-1)}^{1/2} - 1 > {|V_{1}|}^{1/2}$$
by applying this latter fact to the right-hand side of inequality
(\ref{ee3}).
\end{proof}

\begin{lemma}
\label{l6} If $|V_{1,1}| = p \geq 17$, then $s_{1} \not\in \{ 2, 3,
4, 5 \}$.
\end{lemma}

\begin{proof}
If $|V_{1,1}| = p \geq 17$ and $s_{1} \in \{ 2, 3, 4, 5 \}$, then
$$|K_{1}/B_{1}| \cdot (p - k(K_{1}/B_{1})) \leq s_{1}! \cdot (p-1)<
 2 \cdot {(p^{s_{1}}-1)}^{1/2} - 1,$$ violating Lemma~\ref{l5}.
\end{proof}

\begin{lemma}
If $|V_{1,1}| = p \geq 17$, then $s_{1}$ cannot be at least $6$.
\end{lemma}

\begin{proof}
Lemma \ref{l5} gives
$${|K_{1}/B_{1}|}^{2} \cdot p^{2} > {|K_{1}/B_{1}|}^{2}
\cdot {(p - k(K_{1}/B_{1}))}^{2} \geq 4 \cdot (|V_{1}|-1) > p^{s_{1}},$$ that is,
$${|K_{1}/B_{1}|}^{2} > p^{s_{1}-2},$$ which implies
\begin{equation}
\label{ee4} 2.5 \cdot \log \Big( {|K_{1}/B_{1}|}^{1/(s_{1}-1)} \Big)
\geq 2 \cdot \Big( \frac{s_{1}-1}{s_{1}-2} \Big) \log \Big(
{|K_{1}/B_{1}|}^{1/(s_{1}-1)} \Big) > \log p
\end{equation}
since $s_{1} \geq 6$.

If $K_{1}/B_{1}$ has no non-cyclic alternating composition factor or
if $K_{1}/B_{1}$ is a primitive permutation group not containing
$\Alt(s_{1})$, then $|K_{1}/B_{1}| \leq 24^{(s_{1}-1)/3}$ by
Lemma~\ref{l25}. This contradicts (\ref{ee4}) since $p \geq 17$.

Let $d \geq 5$ be the largest degree of an alternating composition
factor $S$ of $K_{1}/B_{1}$. Clearly, $s_{1} \geq d$. We may thus
modify (\ref{ee4}) to
\begin{equation}
\label{ee45} 2 \cdot \Big( \frac{d-1}{d-2} \Big) \log \Big(
{|K_{1}/B_{1}|}^{1/(s_{1}-1)} \Big) > \log p.
\end{equation}
Applying the estimate $|K_{1}/B_{1}| \leq d!^{(s_{1}-1)/(d-1)} <
d^{(s_{1}-1)(d-2)/(d-1)}$ from Lemma~\ref{l25} to (\ref{ee45}), we
obtain
\begin{equation}
\label{ee455} d^{2} > {d!}^{2/(d-2)} > p.
\end{equation}

It was noted before that the number $k(K_{1}V_{1})$ of complex
irreducible characters of the group $K_{1}V_{1}$ is, by Clifford's
theorem, at least $k(K_{1}) \geq k(K_{1}/B_{1})$ plus the number of
orbits of $K_{1}$ on the set $\mathrm{Irr}(V_{1}) \setminus \{
1_{V_{1}} \}$, where $1_{V_{1}}$ denotes the trivial character of
the normal subgroup $V_{1}$ of $K_{1}V_{1}$.

Now $k(K_{1}/B_{1}) \geq k^{*}(S)$ by \cite[Lemma~2.5]{Pyber}. Since
$\mathrm{Aut}(S) = \Sym(d)$ for every $d \geq 5$ different from $6$,
it is easy to see that $k^{*}(S)$ is equal to the number $\pi'(d)$
of partitions of $d$ with sign $1$, unless $d=6$ when $k^{*}(S) =
\pi'(d) - 1$.

The number of orbits of $K_{1}$ on the set $\mathrm{Irr}(V_{1})
\setminus \{ 1_{V_{1}} \}$ is equal, by Brauer's permutation lemma,
to the number of orbits of $K_{1}$ on $V \setminus \{ 1 \}$, which
is at least $s_{1}$. Since $K_{1}/B_{1}$ is a permutation group
having an alternating composition factor $S$ of degree $d \geq 5$,
we infer that $s_{1} \geq d$.

The previous three paragraphs imply $k(K_{1}V_{1}) \geq \pi'(d) +
s_{1}$ unless $d=6$ and $B_{1} = 1$. If $d = 6$ and $B_{1} = 1$,
then $|K_{1}| \leq 6!^{(s_{1}-1)/5} < 4^{s_{1}-1}$ by
Lemma~\ref{l25} and so the number of orbits of $K_{1}$ on $V
\setminus \{ 1 \}$ is at least
$$\frac{p^{s_{1}}-1}{4^{s_{1}-1}} \geq \frac{17^{s_{1}}-1}{4^{s_{1}-1}}
\geq s_{1}+1.$$ We conclude that $k(K_{1}V_{1}) \geq \pi'(d) +
s_{1}$ in all cases.

Writing this into (\ref{ee455}), we obtain
\begin{equation}
\label{ee45555} {d!}^{2/(d-2)} > p \geq \pi'(d) + s_{1} \geq \pi'(d)
+ d.
\end{equation}

Using statement (1) of Lemma~\ref{l3} and (\ref{ee455}) we get
\begin{equation}
\label{ee5} d^{2} > \pi(d)/(2 \cdot |\mathrm{Out}(S)|).
\end{equation}
It follows that $\pi(d) \geq e^{2 \sqrt{d}}/14$ by
\cite[Corollary~3.1]{M}. Applying this to inequality (\ref{ee5}), we
obtain $d^{2} > e^{2 \sqrt{d}}/(28 \cdot |\mathrm{Out}(S)|)$ forcing
$d < 32$. Furthermore, using Gap \cite{GAP}, we find that inequality
(\ref{ee5}) fails for every $d$ with $27 \leq d < 31$. Thus we may
assume that $d$ satisfies $5 \leq d \leq 26$. This and
(\ref{ee45555}) forces $d \leq 12$ and $p \leq 53$.

We claim that $s_{1} < 2d$. For a contradiction assume that $s_{1}
\geq 2d$. We may modify (\ref{ee45}) and (\ref{ee455}) to
\begin{equation}
\label{ee45'} 2 \cdot \Big( \frac{2d-1}{2d-2} \Big) \log \Big(
{d!}^{1/(d-1)} \Big) > \log p.
\end{equation}
Using the estimate $p \geq \pi'(d) + s_{1} \geq \pi'(d) + 2d$ from
(\ref{ee45555}), inequality (\ref{ee45'}) becomes
\begin{equation}
\label{ee45''} {d!}^{(2d-1)/{(d-1)}^{2}} > p \geq \pi'(d) + 2d.
\end{equation}
There is no pair $(d,p)$ with $5 \leq d \leq 12$ and $p$ a prime
satisfying (\ref{ee45''}). We conclude that $s_{1} < 2d$.

Recall that $K_{1}/B_{1}$ is a transitive permutation group of
degree $s_{1}$ such that $K_{1}/B_{1}$ has an alternating
composition factor of degree $d$ with $5 \leq d \leq 12$. Since
$s_{1} < 2d$, the group $K_{1}/B_{1}$ cannot be an imprimitive
permutation group. Thus $K_{1}/B_{1}$ is a primitive permutation
group. We observed above that $K_{1}/B_{1}$ must then contain
$\Alt(s_{1})$. We conclude that $K_{1}/B_{1}$ is equal to $\Alt(s_{1})$ or
$\Sym(s_{1})$.

Since we now know that $d = s_{1}$ and $K_{1}/B_{1}$ is equal to
$\Alt(s_{1})$ or $\Sym(s_{1})$, we may modify (\ref{ee45555}) to
\begin{equation}
\label{ee45555'''} {|K_{1}/B_{1}|}^{2/(d-2)} > p \geq k(K_{1}/B_{1})
+ d \ \ \mathrm{with} \ \ K_{1}/B_{1} \in \{ \Alt(d), \Sym(d) \}.
\end{equation}

Let $K_{1}/B_{1} = \Sym(d)$. A Gap \cite{GAP} computation shows that
if a pair $(d,p)$ satisfies (\ref{ee45555'''}) with $6 \leq d \leq
12$ and $p$ a prime, then $$(d,p) \in \{ (8,31), (7,29), (7,23),
(6,23), (6,19), (6,17) \}.$$ Any of these exceptional cases
contradicts Lemma \ref{l5}. Thus $K_{1}/B_{1} = \Alt(d)$, and
therefore if $(d,p)$ is a pair satisfying (\ref{ee45555'''}) with $6
\leq d \leq 12$ and $p \geq 17$ a prime, then $$(d,p) \in \{ (9,31),
(9,29), (8,23), (7,19), (7,17), (6,17) \}.$$ Again, any of these
exceptional cases contradicts Lemma~\ref{l5}.
\end{proof}

\subsection{Inducing from a subgroup which acts as a metacyclic group}
\label{Section 10}

Define the integer $n_{1,1}$ by $|V_{1,1}| = p^{n_{1,1}}$. Observe
that $n_{1,1} > 1$ by Lemma \ref{l1} and the previous subsection.
Assume that $K_{1,1} \leq \mathrm{\Gamma L}(1,p^{n_{1,1}})$. Notice
that $s_{1} \geq 2$ by Lemma~\ref{metacyclic}.

The normal subgroup $B_{1}$ of $K_{1}$ may be viewed as a subgroup
of the direct product of $s_{1}$ copies of $\mathrm{\Gamma
L}(1,p^{n_{1,1}})$. In particular, $B_{1}$ has an abelian subgroup
$A_{1}$ of index at most $n_{1,1}^{s_{1}}$. Thus
\begin{equation}
\label{A_1} p \geq k(K_{1}V_{1}) \geq k(K_{1}) \geq
\frac{|A_{1}|}{|K_{1}:A_{1}|} \geq \frac{|A_{1}|}{n_{1,1}^{s_{1}}
\cdot |K_{1}/B_{1}|}
\end{equation}
by Lemma~\ref{l4} and by a result of Ernest \cite[p.~502]{Ernest}
saying that whenever $Y$ is a subgroup of a finite group $X$ then
$k(Y)/|X:Y| \leq k(X)$.

Recall that $K_{1}/B_{1}$ may be viewed as a permutation group of
degree $s_{1}$. Since $k(K_{1}/B_{1}) < p$ by Lemma~\ref{l4}, it
follows that $|K_{1}/B_{1}| \leq {(\log p)}^{2(s_{1}-1)}$ by Lemma
\ref{l3}. Thus
\begin{equation}
\label{A_1'} |A_{1}| \leq p \cdot n_{1,1}^{s_{1}} \cdot {(\log
p)}^{2(s_{1}-1)}
\end{equation}
by (\ref{A_1}).

Since $k(K_{1}V_{1}) \leq p$ by Lemma~\ref{l4}, the group $K_{1}$
has at most $p$ orbits on $V_{1}$. In particular, $|V_{1}|/|K_{1}|
\leq p$, which implies
\begin{equation}
\label{A_1''} p^{n_{1,1} \cdot s_{1}} = |V_{1}| \leq |A_{1}| \cdot p
\cdot n_{1,1}^{s_{1}} \cdot {(\log p)}^{2(s_{1}-1)}.
\end{equation}
We get
\begin{equation}
\label{A_1'''} p^{n_{1,1} \cdot s_{1}} \leq {\Big( p \cdot
n_{1,1}^{s_{1}} \cdot {(\log p)}^{2(s_{1}-1)} \Big)}^{2}
\end{equation}
by (\ref{A_1'}) and (\ref{A_1''}). Since $p^{n_{1,1}}/n_{1,1}^{2}$
is smallest when $n_{1,1} = 2$, for fixed $p \geq 17$ and with
$n_{1,1} \geq 2$, we certainly have
\begin{equation}
\label{B_1} p^{2 \cdot s_{1}} \leq {\Big( p \cdot 2^{s_{1}} \cdot
{(\log p)}^{2(s_{1}-1)} \Big)}^{2}
\end{equation}
by (\ref{A_1'''}). Inequality (\ref{B_1}) is equivalent to
\begin{equation}
\label{B_1'} {\Big( \frac{p^{2}}{4 {(\log p)}^{4}} \Big)}^{s_1} \leq
\frac{p^{2}}{{(\log p)}^{4}}.
\end{equation}
Assume that $p > 256$. Then $p > 4 {(\log p)}^{2}$. Since the base
of the power on the left-hand side of (\ref{B_1'}) is larger than
$1$, the left-hand side takes its minimum at $s_{1} = 2$, for any
fixed prime $p$ larger than $256$. But (\ref{B_1'}) fails for $s_{1}
= 2$ and $p > 256$.

\subsection{An explicit constant when $H$ is solvable}

\begin{theorem}
\label{primes1} Let $V$ be a vector space of order at least $p^{3}$
defined over a field of characteristic $p > 7200$. If $H \leq
\GL(V)$ is a finite solvable group of order prime to $p$,
then $k(HV) > p$.
\end{theorem}

\begin{proof}

Let $n > 2$ and let $H \leq \GL(V)$ be of order prime to $p$
such that $k(HV) \leq p$. Let us use the notation of
Subsection~\ref{Section 1}.

Assume that some $V_{i,j}$ has order at least $p^{3}$. Let $n_{i,j}$
be such that $|V_{i,j}| = p^{n_{i,j}}$. Then $K_{i,j}$ cannot be a
subgroup of $\mathrm{\Gamma L}(1, p^{n_{i,j}})$ by
Subsection~\ref{Section 10}, since $p > 256$. Otherwise $K_{i,j}$
has more than $p$ orbits on $V_{i,j}$, since $p > 3^{6}$, which also
contradicts $k(HV) \leq p$. Thus every $V_{i,j}$ has order $p$ or
$p^{2}$.

Let $|V_{i,j}| = p^{2}$. If $K_{i,j} \leq \mathrm{\Gamma
L}(1,p^{2})$, then $s_{i} = 1$ by Subsection \ref{Section 10}.
Moreover, the case $K_{i,j} \leq \mathrm{\Gamma L} (1,p^{2})$ and
$s_{i} = 1$ cannot occur by Lemma \ref{l4} and
Lemma~\ref{metacyclic}. Otherwise $K_{i,j}$ has order less than
$60p$ and so has more than $p/60$ orbits on $V_{i,j}$. In this case
the number of orbits of $H$ on $V_{i}$ is at least
$\binom{s_{i}+m-1}{m-1}$ by \cite[Lemma~2.6]{F} where $m$ is the
number of orbits of $H_{i,j}$ on $V_{i,j}$. If $s_{i} \geq 2$, this
is more than $m^{2}/2 \geq p^{2}/7200 > p$, provided that $p >
7200$, contradicting the fact that $H$ must have less than $p$
orbits on $V$. We conclude that $|V_{i,j}| = p^{2}$ implies $s_{i} =
1$.

If $|V_{i,j}| = p$, then $s_{i} = 1$ by Subsection~\ref{Section 9}.

We proved not only that each $s_{i}$ is equal to $1$ but the
previous argument also gives that whenever $i_{1}$ and $i_{2}$ are
indices with $|V_{i_{1},1}| = |V_{i_{2},1}| = p^{2}$ then $i_{1} =
i_{2}$.

We claim that whenever $i_{1}$ and $i_{2}$ are indices with
$|V_{i_{1},1}| = |V_{i_{2},1}| = p$ then $i_{1} = i_{2}$. Assume
that this is not the case and consider the faithful action $K$ of
$H$ on $W = V_{1,1} + V_{2,1}$. Then $k(KW) \leq p$ by the proof of
Lemma~\ref{l4}.  Since $K$ is an abelian group (of exponent dividing
$p-1$), we have
$$p \geq k(KW) \geq |K| + \frac{|W|-1}{|K|} \geq 2 {(|W|-1)}^{1/2} = 2 {(p^{2}-1)}^{1/2}.$$
This is a contradiction.

Since we are assuming that $n \geq 3$, the $H$-module $V$ must now
be a direct sum of a module $V_{1,1}$ of size $p$ and a module
$V_{2,1}$ of size $p^{2}$. Moreover $K_{2,1}$ acts primitively on
$V_{2,1}$ and $|K_{2,1}/Z(K_{2,1})| \leq 60$. It is easy to see that
$H$ has an abelian normal subgroup $A$ (of order at most
${(p-1)}^{2}$) of index at most $60$. Thus
$$p \geq k(HV) \geq \frac{|A|}{60} + \frac{|V|-1}{60 |A|}
\geq \frac{1}{30} {(p^{3}-1)}^{1/2}.$$ This is again a contradiction
since $p > 1000$.
\end{proof}

\section{Proof of Theorem~\ref{theorem-general-bound-p^3} and
further remarks}   \label{section-proof-1.4}

We can finally prove Theorem~\ref{theorem-general-bound-p^3}, which
is restated below.

\begin{theorem}
Let $G$ be a finite group, $p$ a prime and $P$ a Sylow $p$-subgroup of $G$.
 If $|P/\Phi(P)|\geq p^3$, then $|\Irr_{p',\parat}(G)|>p$ provided that any of
 the following two conditions holds.
 \begin{itemize}
  \item[(1)] $G$ is solvable and $p > 7200$; or
  \item[(2)] the McKay--Navarro conjecture is true and $p$ is sufficiently
   large.
 \end{itemize}
\end{theorem}

\begin{proof}
As mentioned in Section~\ref{section-mckay-navarro}, since the
McKay--Navarro conjecture is known to be true for solvable groups, we
see that the number of almost $p$-rational irreducible characters of
$p'$-degree of $G$ is $|\Irr_\parat(\bN_G(P)/P')|$, which is at
least $|\Irr_\parat(\bN_G(P)/\Phi(P))|$. As every irreducible
character of $\bN_G(P)/\Phi(P)$ is almost $p$-rational, it follows
that the number of almost $p$-rational irreducible characters of
$p'$-degree of $G$ is at least $k(\bN_G(P)/\Phi(P))$. Since
$|P/\Phi(P)|$ is divisible by $p^3$, this class number
$k(\bN_G(P)/\Phi(P))$ of $\bN_G(P)/\Phi(P)$ is greater than $p$ by
Theorem~\ref{primes1}, and thus the first part of the theorem is
proved.

The second part follows in the same way, but using
Theorem~\ref{IulianAttila} instead.
\end{proof}

We now prove the statement following Theorem~\ref{theorem-p'-degree}
in the introduction.

\begin{theorem}\label{thm:GHSV}
Let $G$ be a non-trivial finite group and $p$ and $q$ be (possibly
equal) primes. Then $G$ possesses a non-trivial irreducible character
that is of $\{p,q\}'$-degree and almost $\{p,q\}$-rational.
\end{theorem}

\begin{proof} It was proved in \cite[Theorem~2.1]{Giannelli-Hung-Schaeffer-Rodriguez}
that, for every nonabelian simple group $S$ and every set of primes
$\pi=\{p,q\}$, there exists $\mathbf{1}_S\neq\chi\in \Irr(S)$ of
$\pi'$-degree such that $\QQ(\chi)\subseteq \QQ(e^{2\pi i/p})$ or
$\QQ(\chi)\subseteq \QQ(e^{2\pi i/q})$, unless $(S,\pi)$ is
$(\tw{2}F_4(2)',\{3,5\})$, $(J_4,\{23,43\})$, or $(J_4,\{29, 43\})$.
One can check from \cite{Atl1} that, for these exceptions, there
still exists $\mathbf{1}_S\neq\chi\in \Irr(S)$ of $\pi'$-degree such
that $\chi$ is both almost $p$-rational and $q$-rational. The same
arguments as in the proof of
\cite[Theorem~C]{Giannelli-Hung-Schaeffer-Rodriguez} then yield the
conclusion.
\end{proof}

We put forward the following, which is based on the McKay--Navarro
conjecture and another well-known conjecture that the number of
conjugacy classes of any finite group is bounded below
logarithmically by the order of the group.

\begin{conjecture}\label{conjecture}
There exists a universal constant $c > 0$ such that whenever $G$ is
a finite group and $P$ is a Sylow $p$-subgroup of $G$ then the
number of almost $p$-rational irreducible characters of $p'$-degree
of $G$ is greater than $c \cdot \log_2(|P/\Phi(P)|)$.
\end{conjecture}

Even the weaker statement that $|\Irr_{p',\parat}(G)|\rightarrow
\infty$ as $|P/\Phi(P)|\rightarrow \infty$ seems non-trivial to us
(in the case when $G$ is not $p$-solvable). By
Theorem~\ref{theorem-p'-degree}, this is reduced to showing that
$|\Irr_{p',\parat}(G)|\rightarrow \infty$ as the minimum number of
generators of $P$ approaches infinity.

We conclude this section by remarking that, as $p$-rationality and
almost $p$-rationality of irreducible characters can be seen from
the character table, it would be interesting to know a
group-theoretic characterization of groups having the property that
all irreducible characters are (almost) $p$-rational for a fixed
prime $p$. Let $\sigma_e$ ($e\geq 1$) be the Galois automorphism in
$\Gal(\QQ_{|G|}/\QQ)$ that fixes $p'$-roots of unity and sends every
$p$-power root of unity to its $(1+p^e)$-th power. Navarro and Tiep
\cite{Navarro-Tiep} proved a consequence of the McKay--Navarro
conjecture that if all $p'$-degree irreducible characters of $G$ are
$\sigma_e$-fixed, then $P/P'$ has exponent at most $p^e$, where
$P\in\Syl_p(G)$. Therefore, if all irreducible $p'$-characters of
$G$ are almost $p$-rational, then $P/P'$ is elementary abelian. The
converse is expected to be also true, and has been reduced to the
same statement for almost quasi-simple groups in
\cite[Theorem~C]{Navarro-Tiep}, which in turns has been solved for
$p=2$ in \cite{Malle19}.

\section{Almost $p$-rational characters and cyclic Sylow
$p$-subgroups}\label{sect:last}

In this last section we make some remarks on
Question~\ref{conj:cyclicity}, which predicts that Sylow
$p$-subgroups of a finite group $G$ of order divisible by $p$ are
cyclic if and only if \[|\Irr_{p',\parat}(B_0(G))|\in
\cS_p:=\{e+\frac{p-1}{e}: e\in\ZZ^+, e\mid p-1\}.\]

Note that, when $P\in\Syl_p(G)$ is cyclic, the
Alperin--McKay--Navarro conjecture is known to be true
\cite{Navarro04}, and thus $|\Irr_{p',\parat}(B_0(G))|$ is the class
number of a semidirect product of a certain $p'$-group acting
faithfully on $P/\Phi(P)\cong C_p$, which then belongs to the set
$\cS_p$. The `only if' implication therefore easily follows.

Question~\ref{conj:cyclicity} is related to the following, which came
out of the results in Sections~\ref{section-affine-groups} and
\ref{section-affine-groups-2}.

\begin{question}   \label{conj:k(HV)}
 Let $H$ be a $p'$-group acting faithfully on a finite vector space~$V$ of
 size~$p^n$. Is it true that $k(HV)\notin \cS_p$ whenever $n\ge 2$?
\end{question}

We are able to answer this in some cases.

\begin{theorem}   \label{thm:cyclicity-p-large}
 Question \ref{conj:k(HV)} has an affirmative answer, provided that any of the following
 conditions hold.
 \begin{itemize}
  \item[(1)] $p < 17$;
  \item[(2)] $p$ is sufficiently large; or
  \item[(3)] the group $H$ is solvable and $p > 7 300 000$.
 \end{itemize}
\end{theorem}

\begin{proof}
Let $p$ be a prime and let $H$ be a $p'$-group acting faithfully on
a finite vector space $V$ of size $p^n$.

Assume that (1) holds. The result follows from Lemma~\ref{l1}, the observation
of Navarro that $k(HV) = 10$ when $HV = (C_{11} \times C_{11}):\SL_2(5)$ and by
noting that $10 \not\in \cS_{11}$.

We may assume that $n = 2$ by Theorem~\ref{IulianAttila} (if $p$ is
sufficiently large) and by Theorem~\ref{primes1} (if $H$ is solvable and
$p > 7200$).

Assume that $k(HV) \leq p$ (and $p \geq 17$). It follows that $H$ is primitive
and irreducible on $V$ by Lemma~\ref{primitiveandirreducible}.

Let $H$ be solvable. Assume that $p > 270 000$. Then
(\ref{lowerupper2}) from the proof of Lemma~\ref{l2} is
$$\Big( \frac{8}{m} + \frac{m}{48} \Big) p - 4 <
  k(HV) < \Big( \frac{8}{m} + \frac{m}{48} \Big) p + 7201,$$
for some even integer $m$ with $12 \leq m \leq 36$.

Let $H$ be non-solvable. Assume that $p > 7 300 000$. Then~(\ref{60'}) from the
proof of Lemma~\ref{A_5} is
$$\frac{c}{m} (p-1) + \frac{m}{60} (p+1) \leq k(HV) \leq
  \frac{c}{m} (p-1) + \frac{m}{60} (p+1) + 7200,$$
where $m$ divides $p-1$ such that $5 \leq m \leq 55$ (and $(p-1)/m$ is even)
or $12\leq m \leq 48$ (and $(p-1)/m$ is odd) and $c=9$ if $|Z(H)|$ is odd
and $c = 4.5$ if $|Z(H)|$ is even.

It follows from (\ref{lowerupper2}) and (\ref{60'}) together with a
bit of computer calculation that
$$\frac{263}{480} p - 4 < k(HV) < \frac{659}{660}p + 7201.$$
For $p > 5 \cdot 10^{6}$, we have $\frac{263}{480}p - 4 > \frac{p-1}{2} + 2$
and $\frac{659}{660} p + 7201 < p$. The result follows.
\end{proof}

\begin{theorem}
 An affirmative answer to Question~\ref{conj:k(HV)} and
 the principal block case of the Alperin--McKay--Navarro imply an
 affirmative answer to Question~\ref{conj:cyclicity}.
\end{theorem}

\begin{proof}
Assume that both the statement in Question~\ref{conj:k(HV)} and the
principal block case of the Alperin--McKay--Navarro conjecture
hold true, and that $|\Irr_{p',\parat}(B_0(G))|\in \cS_p$. As
discussed in Section~\ref{section-mckay-navarro}, we then have
\[|\Irr_{p',\parat}(B_0(\bN_G(P)))|\in \cS_p,\]
where $P\in\Syl_p(G)$. By Fong's theorem (see
\cite[Theorem~10.20]{Navarro}), it follows that
\[|\Irr_{p',\parat}(\bN_G(P)/\bO_{p'}(\bN_G(P)))|\in
\cS_p.\] Since $\chi\in\Irr(\bN_G(P)/\bO_{p'}(\bN_G(P)))$ is almost
$p$-rational and of $p'$-degree if and only if $\chi$ lies over some
$\theta\in\Irr(P/\Phi(P))$ (by Lemma~\ref{lemma-NT}), we now have
\[|\Irr(\bN_G(P)/\Phi(P)\bO_{p'}(\bN_G(P)))|\in
\cS_p.\] It then follows that $\dim(P/\Phi(P))=1$, and therefore $P$
is cyclic, as desired.
\end{proof}

As the Alperin--McKay--Navarro conjecture is known for $p$-solvable groups
(see Section~\ref{section-mckay-navarro}), we have the following for now.

\begin{theorem}
Let $p$ be a sufficiently large prime. Then for any finite $p$-solvable group $G$ of order divisible by $p$, the Sylow $p$-subgroups of $G$ are
cyclic if and only if $|\Irr_{p',\parat}(B_0(G))|\in \cS_p$.
\end{theorem}


\end{document}